\documentclass{amsart}
\usepackage{amsfonts, amsmath, amssymb, amscd, amsthm}
\usepackage{amsbsy,bm}

\theoremstyle{plain}
\newtheorem{theorem}{Theorem}[section]
\newtheorem*{theorem*}{Main Theorem}
\newtheorem{lemma}{Lemma}[section]

\newtheorem{proposition}{Proposition}[section]

\theoremstyle{definition}

\newtheorem{examples}{Examples}[section]
\newtheorem{remark}{Remark}[section]

\newcommand{\DP}[1]{(\tilde{\nabla}_{#1}P)}

\numberwithin{equation}{section}

\normalsize
\begin{document}

\title[On Hopf hypersurfaces of the nearly K\"ahler $\mathbf{S}^3\times\mathbf{S}^3$]
{On Hopf hypersurfaces of the homogeneous nearly K\"ahler $\mathbf{S}^3\times\mathbf{S}^3$}

\author[Zejun Hu and Zeke Yao]{Zejun Hu and Zeke Yao}

\address{School of Mathematics and Statistics, Zhengzhou University,
Zhengzhou 450001, People's Republic of China}
\email{huzj@zzu.edu.cn; yaozkleon@163.com}


\thanks{This project was supported by NSF of China, Grant Number 11771404.}


\thanks{{\sl School of Mathematics and Statistics, Zhengzhou University,
Zhengzhou 450001, People's Republic of China}.\ \
{\bf E-mails}: huzj@zzu.edu.cn; yaozkleon@163.com}

\begin{abstract}
In this paper, extending our previous joint work (Hu et al., Math Nachr 291:343--373, 2018),
we initiate the study of Hopf hypersurfaces in the
homogeneous NK (nearly K\"ahler) manifold $\mathbf{S}^3\times\mathbf{S}^3$. First,
we show that any Hopf hypersurface of the homogeneous NK $\mathbf{S}^3\times\mathbf{S}^3$
does not admit two distinct principal curvatures. Then, for the important class
of Hopf hypersurfaces with three distinct principal curvatures, we establish a
complete classification under the additional condition that their
holomorphic distributions $\{U\}^\perp$ are preserved by the almost product
structure $P$ of the homogeneous NK $\mathbf{S}^3\times\mathbf{S}^3$.
\end{abstract}

\maketitle

%========================================
\section{Introduction}\label{sect:1}
Let $\bar{M}$ be an almost Hermitian manifold with almost complex structure $J$.
Given a connected orientable real hypersurface $M$ of $\bar{M}$, there appears an important
notion the {\it structure vector field} defined by $U:=-J\xi$, where $\xi$ is the unit
normal vector field. If the integral curves of $U$ are geodesics, then it is well known
that $M$ is called a {\it Hopf hypersurface}. During the last four decades, Hopf hypersurfaces of the complex space forms and
several other almost Hermitian manifolds have been extensively and deeply investigated,
for details we refer to \cite{B,CR,KM,M,NR} and \cite{BBW,BS,HJZ} and the references
therein.
Recall that a nearly K\"ahler (NK) manifold is an almost Hermitian manifold such that
the covariant derivative of the almost complex structure $J$ is skew-symmetric. It is
well known from Nagy's classification of nearly K\"ahler manifolds \cite{N} that
the six-dimensional ones are important construction factor, and from Butruille \cite{But1,
But2} that the only homogeneous $6$-dimensional NK manifolds are the $6$-sphere
$\mathbf{S}^6$, the $\mathbf{S}^3\times\mathbf{S}^3$, the complex projective space
$\mathbf{CP}^3$ and the flag manifold $SU(3)/U(1)\times U(1)$, and moreover from
Foscolo-Haskins \cite{F-H} that both $\mathbf{S}^6$ and $\mathbf{S}^3\times\mathbf{S}^3$
admit inhomogeneous NK structures.

Notice that the Riemannian geometric invariants of the homogeneous NK $\mathbf{S}^3\times
\mathbf{S}^3$ were systematically presented by Bolton-Dillen-Dioos-Vrancken \cite{B-D-D-V}.
Since then the study of the canonical submanifolds of the homogeneous NK $\mathbf{S}^3\times\mathbf{S}^3$
becomes quite active and many interesting results have been obtained. This includes the results
about almost complex surfaces in \cite{B-D-D-V,D-L-M-V,H-Z1}, about Lagrangian and CR submanifolds
in \cite{ADMV,B-M-J-L1, B-M-J-L2, D-V-W, H-Z2, Z-D-H-V-W}. Nevertheless, about hypersurfaces the
results are few that appear only in \cite{H-Y-Z,H-Y-Z2}.

The goal of this paper is to study Hopf hypersurfaces in the homogeneous NK $\mathbf{S}^3
\times\mathbf{S}^3$. In this situation, according to Proposition 1 of \cite{BBW}, the Hopf
condition is equivalent to that the structure vector field is a principal curvature vector
field of the hypersurface.

Our first concern is Hopf hypersurfaces with two distinct principal curvatures. The result we
obtain is the following:
\begin{theorem}\label{thm:1.1}
No Hopf hypersurface in the homogeneous NK $\mathbf{S}^3\times\mathbf{S}^3$ admits
exactly two distinct principal curvatures.
\end{theorem}

Our next concern is Hopf hypersurfaces with three distinct principal curvatures. It turns out that
hypersurfaces of this class are quite complicated and examples of at least three families appear.
As the second main result of this paper, we obtain a classification of them under the additional/natural
condition that
their holomorphic distributions $\{U\}^\perp$ are preserved by the almost product structure $P$ of
the homogeneous NK $\mathbf{S}^3\times\mathbf{S}^3$.
Before stating the result, we would recall that, according to Moruz-Vrancken \cite{M-V} and Podest\`a-Spiro
\cite{P-S}, the following three maps
\begin{enumerate}
\item[(1)]
$\mathcal{F}_1:\ \mathbf{S}^3 \times \mathbf{S}^3 \rightarrow
          \mathbf{S}^3\times \mathbf{S}^3{\rm\ with\ }
          \mathcal{F}_1(p,q)=(q,p),$
\item[(2)]
$\mathcal{F}_2:\ \mathbf{S}^3 \times \mathbf{S}^3 \rightarrow
          \mathbf{S}^3\times \mathbf{S}^3{\rm\ with\ }
          \mathcal{F}_2(p,q)=(\bar{p},q\bar{p}),$
\item[(3)]
$\mathcal{F}_{abc}:\ \mathbf{S}^3 \times \mathbf{S}^3 \rightarrow
          \mathbf{S}^3\times \mathbf{S}^3{\rm\ with\ }
          \mathcal{F}_{abc}(p,q)=(ap\bar c, bq\bar c)$ for any unitary quaternions $a,b,c$
\end{enumerate}
are isometries of the NK $\mathbf{S}^3 \times \mathbf{S}^3$. Then, the result can be stated as follows:

\begin{theorem}\label{thm:1.2}
Let $M$ be a Hopf hypersurface of the homogeneous NK $\mathbf{S}^3\times\mathbf{S}^3$ with
three distinct principal curvatures. If $P\{U\}^\perp=\{U\}^\perp$, then, up to isometries
of type $\mathcal{F}_{abc}$, $M$ is locally given by one of the following embeddings
$f_r, f_r'$ and $f_r'': \mathbf{S}^3\times\mathbf{S}^2\rightarrow\mathbf{S}^3\times\mathbf{S}^3$
defined by:
$$
f_r(x,y)=(x,\sqrt{1-r^2}+ry),\ \ \ \ f_r'=\mathcal{F}_1\circ f_r,\ \ \ \ f_r''=\mathcal{F}_2\circ f_r,
$$
where $0<r\leq1$, $x\in\mathbf{S}^3$, $y\in\mathbf{S}^2\subset\mathbb{R}^3$, and as usual
$\mathbf{S}^3$ (resp. $\mathbf{S}^2$) is regarded as the set of the unitary (resp. imaginary)
quaternions in the quaternion space $\mathbb{H}$.
\end{theorem}

\begin{remark}\label{rem:1.1}
Let $M_1^{(r)},M_2^{(r)},M_3^{(r)}$ denote the images of the three embeddings $f_r, f_r',f_r''$, respectively.
Then, for $0<r\le1$, $M_1^{(r)},M_2^{(r)}$ and $M_3^{(r)}$ correspond to the
three possibilities of the action $P$ on the unit normal vector field $\xi$,
which we shall establish in Proposition \ref{prop:5.1} below.
\end{remark}

\begin{remark}\label{rem:1.2}
Theorem \ref{thm:1.2} is an extension of the previous result in \cite{H-Y-Z}, where the
hypersurfaces $M_1^{(r)},M_2^{(r)},M_3^{(r)}$ corresponding to $r=1$ were characterized by the property
of satisfying $A\phi=\phi A$, where $A$ is the shape operator of the hypersurfaces and
$\phi$ is the almost contact structure induced from $J$. Moreover, it is worthy to mention
that each of the hypersurfaces $M_1^{(r)},M_2^{(r)}$ and $M_3^{(r)}$ is minimal if and only if $r=1$.
\end{remark}

\begin{remark}\label{rem:1.3}
Theorem \ref{thm:1.2} shows that Niebergall and Ryan's observation (cf. p.234 of \cite{NR}),
which states that certain interesting classes of hypersurfaces in the complex space forms
can be characterized by conditions on the holomorphic distribution $\{U\}^\perp$, is similarly
valid for the homogeneous NK $\mathbf{S}^3\times\mathbf{S}^3$. On the other hand, at the moment
we do not know if there exist Hopf hypersurfaces of the homogeneous NK
$\mathbf{S}^3\times\mathbf{S}^3$ that have three distinct principal curvatures and satisfy $P\{U\}^\perp\not=\{U\}^\perp$.
\end{remark}

%========================================================
\section{Preliminaries}\label{sect:2}
\subsection{The homogeneous NK structure on $\mathbf{S}^3\times\mathbf{S}^3$}\label{sect:2.1}~

One can look the classical and comprehensive study of the NK manifolds from [14]. In this 
section, we first collect some necessary materials from \cite{B-D-D-V}. Let us denote by
$\mathbf{S}^3$ the $3$-sphere in $\mathbb{R}^4$ as the set of all unitary quaternions.
By the natural identification $T_{(p,q)}(\mathbf{S}^3\times\mathbf{S}^3)\cong T_p
\mathbf{S}^3\oplus T_q\mathbf{S}^3$, we write a tangent vector at $(p,q)\in
\mathbf{S}^3 \times\mathbf{S}^3$ as $Z(p,q)=(U_{(p,q)},V_{(p,q)})$ or simply $Z=(U,V)$.
The well-known almost complex structure $J$ on $\mathbf{S}^3\times\mathbf{S}^3$
is defined by
\begin{equation}\label{eqn:2.1}
J Z(p,q)=\tfrac{1}{\sqrt{3}}(2pq^{-1}V-U, -2qp^{-1}U+V).
\end{equation}

On $\mathbf{S}^3\times\mathbf{S}^3$ we can define a Hermitian metric $g$ compatible with $J$ by
\begin{equation}\label{eqn:2.2}
\begin{split}
g(Z, Z')& = \tfrac{1}{2}(\langle Z,Z'\rangle  + \langle JZ, JZ' \rangle )\\
& = \tfrac{4}{3}(\langle U,U' \rangle + \langle V,V' \rangle)
   - \tfrac{2}{3}(\langle p^{-1}U, q^{-1}V' \rangle + \langle p^{-1}U', q^{-1}V \rangle),
\end{split}
\end{equation}
where $Z=(U, V)$ and $Z'=(U', V')$ are tangent vectors, and
$\langle\cdot,\cdot\rangle$ is the standard product metric on
$\mathbf{S}^3\times\mathbf{S}^3$. Then $\{g,J\}$ gives the
homogeneous NK structure on $\mathbf{S}^3\times\mathbf{S}^3$.

Let $\tilde\nabla$ be the Levi-Civita connection with respect to $g$, and
as usual we define a $(1,2)$-tensor field $G$ by $G(X,Y):=(\tilde\nabla_X J)Y$
for $X,Y\in T(\mathbf{S}^3\times\mathbf{S}^3)$. Then, we have the following
formulas for $G$:
\begin{gather}
G(X,Y)+G(Y,X)=0,\label{eqn:2.3}\\
G(X,JY)+JG(X,Y)=0,\label{eqn:2.4}\\
g(G(X,Y),Z)+g(G(X,Z),Y)=0,\label{eqn:2.5}\\
\begin{aligned}\label{eqn:2.6}
g(G(X,Y),G(Z,W))=&\tfrac{1}{3}\big[g(X,Z)g(Y,W)-g(X,W)g(Y,Z)\\
&\quad +g(JX,Z)g(JW,Y)-g(JX,W)g(JZ,Y)\big].
\end{aligned}
\end{gather}

An almost product structure $P$ on $\mathbf{S}^3\times\mathbf{S}^3$ is
introduced by
\begin{equation}\label{eqn:2.7}
PZ=(pq^{-1}V, qp^{-1}U),\ \ \forall\, Z=(U,V)\in T_{(p,q)}(\mathbf{S}^3\times\mathbf{S}^3).
\end{equation}

It is easily seen that $P$ is compatible with the metric $g$, i.e., $P$ is symmetric with
respect to $g$. Also $P$ is anti-commutative with $J$. Moreover, with respect to $G$ and $P$,
we further have
\begin{equation}\label{eqn:2.8}
2\DP{X}Y=JG(X,PY)+JPG(X,Y),
\end{equation}
\begin{equation}\label{eqn:2.9}
PG(X,Y)+G(PX,PY)=0.
\end{equation}

Note also that in terms of $P$ the usual product structure $Q$, defined by $Q(Z)=(-U,V)$ for $Z=(U,V)$,
can be expressed by
\begin{equation}\label{eqn:2.10}
QZ=\tfrac{1}{\sqrt{3}}(2PJZ-JZ).
\end{equation}

For the NK $\mathbf{S}^3\times\mathbf{S}^3$, we also need the useful relation between
the NK connection $\tilde{\nabla}$ and the usual Euclidean connection $\nabla^E$ (cf. Lemma 2.2 of
\cite{D-L-M-V} and Remark 2.5 of \cite{D-V-W}):
\begin{equation}\label{eqn:2.11}
\nabla_X^E Y=\tilde{\nabla}_X Y+\tfrac12[JG(X,PY)+JG(Y,PX)].
\end{equation}

The Riemannian curvature tensor $\tilde R$ of the NK $\mathbf{S}^3\times\mathbf{S}^3$
is given by
\begin{equation}\label{eqn:2.12}
\begin{split}
\tilde{R}(X,Y)Z=&\tfrac{5}{12}\big[g(Y,Z)X-g(X,Z)Y\big]\\
 &+\tfrac{1}{12}\big[g(JY,Z)JX-g(JX,Z)JY-2g(JX,Y)JZ\big]\\
 &+\tfrac{1}{3}\big[g(PY,Z)PX-g(PX,Z)PY\\
 &\qquad +g(JPY,Z)JPX-g(JPX,Z)JPY\big].
\end{split}
\end{equation}

\subsection{Hypersurfaces of the NK $\mathbf{S}^3\times\mathbf{S}^3$}\label{sect:2.2}~

Let $M$ be a hypersurface of the NK $\mathbf{S}^3\times\mathbf{S}^3$
with unit normal vector field $\xi$. For any vector field $X$ tangent
to $M$, we have the decomposition
\begin{equation}\label{eqn:2.13}
JX=\phi X+\eta(X)\xi,
\end{equation}
where $\phi X$ and $\eta(X)\xi$ are the tangent and normal parts of $JX$,
respectively. Then $\phi$ is a tensor field of type (1,1), $\eta$ is a
$1$-form on $M$. By definition, the following relations hold:
\begin{equation}\label{eqn:2.14}
\left\{
\begin{aligned}
&\eta(X)=g(X,U),\ \ \eta(\phi X)=0,\ \ \phi^2X=-X+\eta(X)U,\ \ \phi U=0,\\
&g(\phi X,Y)=-g(X,\phi Y),\ \ g(\phi X,\phi Y)=g(X,Y)-\eta(X)\eta(Y),
\end{aligned}\right.
\end{equation}
where $U:=-J\xi$ is called the {\it structure vector field} of $M$. The
equations \eqref{eqn:2.14} show that $(\phi,U,\eta,g)$ determines
an {\it almost contact metric structure} over $M$.

Let $\nabla$ be the induced connection on $M$ and $R$ its Riemannian
curvature tensor. The formulas of Gauss and Weingarten state that
\begin{equation}\label{eqn:2.15}
\begin{split}
\tilde \nabla_X Y=\nabla_X Y + h(X,Y),\quad \tilde \nabla_X \xi=-
A X , \ \ \forall\, X,Y \in TM,
\end{split}
\end{equation}
where $h$ is the second fundamental form and $A$ is the shape operator.
They are related by $h(X,Y)=g(AX,Y)\xi$. Using the formulas of Gauss and
Weingarten, we can easily show that
\begin{equation}\label{eqn:2.16}
\nabla_X U=\phi AX-G(X,\xi).
\end{equation}

The Gauss and Codazzi equations of $M$ are given by
\begin{equation}\label{eqn:2.17}
\begin{split}
R(X,Y)Z=&\tfrac{5}{12}\big[g(Y,Z)X - g(X,Z)Y\big]\\
 &+ \tfrac{1}{12}\big[g(JY,Z)\phi X - g(JX,Z)\phi Y - 2g(JX,Y)\phi Z\big]\\
 &+ \tfrac{1}{3}\Big[g(PY,Z)(PX)^\top - g(PX,Z)(PY)^\top\\
 &\qquad + g(JPY,Z)(JPX)^\top - g(JPX,Z)(JPY)^\top\Big]\\&
 +g(AZ,Y)AX-g(AZ,X)AY,
 \end{split}
\end{equation}
and
\begin{equation}\label{eqn:2.18}
\begin{split}
(\nabla_X A)Y-(\nabla_Y A)X=&\tfrac{1}{12}\big[g(X,U)\phi Y - g(Y,U)\phi X - 2g(JX,Y)U\big]\\
&+ \tfrac{1}{3}\Big[g(PX,\xi)(PY)^\top - g(PY,\xi)(PX)^\top\\
 &\qquad + g(PX,U)(JPY)^\top - g(PY,U)(JPX)^\top\Big],
\end{split}
\end{equation}
where $\cdot^\top$ means the tangential part.

Similar to that of the complex space forms, a hypersurface $M$ of the NK
$\mathbf{S}^3\times\mathbf{S}^3$ is a Hopf hypersurface if and only if
the integral curves of its structure vector field $U$ are geodesics,
i.e., $\nabla_U U=0$. We denote by $\alpha$ the principal
curvature function corresponding to the structure vector field $U$, i.e., $AU=\alpha U$.
First of all, we shall present two elementary lemmas for Hopf
hypersurfaces of the NK $\mathbf{S}^3\times\mathbf{S}^3$ as follows:
\begin{lemma}[cf. \cite{H-Y-Z2}]\label{lemma:2.1}
Let $M$ be a Hopf hypersurface in the NK $\mathbf{S}^3
\times \mathbf{S}^3$. Then we have
\begin{equation}\label{eqn:2.19}
\begin{aligned}
\tfrac16&g(\phi X,Y)-\tfrac23\big[g(PX,\xi)g(PY,U)-g(PX,U)g(PY,\xi)\big]\\
&=g((\alpha I-A)G(X,\xi),Y)+g(G((\alpha I-A)X,\xi),Y)\\
&\ \ \ -\alpha g((A\phi +\phi A)X,Y)+2g(A\phi AX,Y),\ \ X,Y\in
\{U\}^{\bot},
\end{aligned}
\end{equation}
where $\{U\}^\perp$ denotes the subdistribution of $TM$ that is
orthogonal to $U$, and $I$ denotes the identity transformation.
\end{lemma}

\begin{lemma}\label{lemma:2.2}
Let $M$ be a Hopf hypersurface in the NK $\mathbf{S}^3\times \mathbf{S}^3$ satisfying
$P\{U\}^\bot=\{U\}^\perp$. Then the function $\alpha$ is
constant.
\end{lemma}
\begin{proof}
By using the Codazzi equation and the symmetry of $A$, we have the calculation
\begin{equation*}
0=g((\nabla_U A)Y-(\nabla_Y A)U,U)=g((\nabla_U A)U,Y)-g((\nabla_Y A)U,U)=-Y\alpha,\ Y\in\{U\}^\perp.
\end{equation*}
It follows that $\nabla\alpha=(U\alpha)\,U$. Then, for $X,Y\in\{U\}^\perp$, we have
\begin{equation}\label{eqn:2.20}
0=X(Y\alpha)-Y(X\alpha)=[X,Y]\alpha=g([X,Y],U)\,U\alpha.
\end{equation}
If $U\alpha\neq0$ holds on some open set, then \eqref{eqn:2.20} implies that
$[X,Y]\in\{U\}^\perp$. Thus $\{U\}^\perp$ is integrable which gives four-dimensional
almost complex submanifolds of the NK $\mathbf{S}^3\times \mathbf{S}^3$. This is
impossible because, according to Lemma 2.2 of \cite{P-S}, any six-dimensional compact
non-K\"ahler NK manifold admits no almost complex four-dimensional submanifold.
Hence $U\alpha=0$ and $\alpha$ is constant.
\end{proof}

\subsection{A canonical distribution related to hypersurfaces of the NK $\mathbf{S}^3\times\mathbf{S}^3$}\label{sect:2.3}~

In order for choosing an appropriate local orthonormal frame of the NK $\mathbf{S}^3\times\mathbf{S}^3$
along its hypersurface $M$, following that in \cite{H-Y-Z2} we consider
$$
\mathfrak{D}(p):={\rm Span}\,\{\xi(p),U(p),P\xi(p),PU(p)\},\ \ p\in M.
$$

It is easily seen that, since $P$ is anti-commutative with $J$, $\mathfrak{D}$ defines
a distribution on $M$ with dimension exact $2$ or $4$, and that it is invariant under
both $J$ and $P$. Along $M$, let $\mathfrak{D}^\perp$ denote the distribution in
$T(\mathbf{S}^3\times\mathbf{S}^3)$ that is orthogonal to $\mathfrak{D}$ at each
$p\in M$. For later's purpose, we shall make some remarks about $\dim\mathfrak{D}$:

(1) If $\dim\mathfrak{D}=4$ holds in an open set, then there exists a
unit tangent vector field $e_1\in \{U\}^\bot$ and functions
$a,b,c$ with $c>0$ such that
\begin{equation}\label{eqn:2.21}
P\xi=a\xi+bU+ce_1,\ \ a^2+b^2+c^2=1.
\end{equation}

Put $e_2=Je_1$. Moreover, from the fact $\dim\,\mathfrak{D}^\perp=2$ and that
$\mathfrak{D}^\perp$ is invariant under the action of both $J$ and $P$, we can
choose a local unit vector field $e_3\in\mathfrak{D}^\perp$ such that $Pe_3=e_3$.
Now, putting $e_4=Je_3$ and $e_5=U$, then $\{e_i\}_{i=1}^5$ is a well-defined
orthonormal basis of $TM$ and, acting by $P$, it has the following properties:
\begin{equation}\label{eqn:2.22}
\left\{
\begin{aligned}
&P\xi=a\xi+ce_1+be_5,\ \ \ \ \ Pe_1=c\xi-ae_1-be_2,\\
&Pe_2=ce_5-be_1+ae_2,\ \ Pe_3=e_3,\\
&Pe_4=-e_4,\ \hspace{18mm} Pe_5=b\xi+ce_2-ae_5.
\end{aligned}\right.
\end{equation}

(2) If $\dim\mathfrak{D}=2$ holds in an open set, then $P\{U\}^\perp=\{U\}^\perp$
and we can write
\begin{equation}\label{eqn:2.23}
P\xi=a\xi+bU,\ \ a^2+b^2=1.
\end{equation}

Now, $\mathfrak{D}^\perp$ is a $4$-dimensional distribution that is
invariant under the action of both $J$ and $P$. Hence, we can choose
unit vector fields $e_1,\, e_3\in\mathfrak{D}^\perp$ such that
$Pe_1=e_1,\, Pe_3=e_3$. Put $e_2=Je_1,\,e_4=Je_3$ and $e_5=U$. In
this way, we obtain an orthonormal basis $\{e_i\}_{i=1}^5$ of $TM$.
However, we would remark that such choice of $\{e_1, e_3\}$ (resp.
$\{e_2, e_4\}$) is unique up to an orthogonal transformation.

%==================================================
\section{The proof of Theorem \ref{thm:1.1}}\label{sect:3}
Suppose on the contrary that $M$ is a Hopf hypersurface in the NK $\mathbf{S}^3
\times\mathbf{S}^3$ which has two distinct principal curvatures, say $\alpha$ and
$\lambda$, with $AU=\alpha U$. We denote by $V_\alpha$ and $V_\lambda$ the corresponding
eigen-distributions. By the continuity of the principal curvature functions, we know
that the dimensions $(\dim V_\alpha,\dim V_\lambda)$ of the two eigen-distributions
have to be one of the four possibilities: (1,4), (2,3), (3,2) and (4,1).

Next, we separate the proof of Theorem \ref{thm:1.1} into the proofs of two lemmas,
depending on the dimension of $\mathfrak{D}$.
\begin{lemma}\label{lemma:3.1}
The case $\dim\mathfrak{D}=4$ does not occur.
\end{lemma}
\begin{proof}
To argue by contradiction we assume that $\dim\mathfrak{D}=4$ does hold on an open set.
Now we check each possibility of $(\dim V_\alpha,\dim V_\lambda)$.

\vskip 1mm

{\bf (i) $(\dim V_\alpha,\dim V_\lambda)=(1,4)$ on $M$.}

In this case, it is easy to see that $A\phi=\phi A$ holds. This is impossible
because, according to Theorem 4.1 of \cite{H-Y-Z}, hypersurfaces satisfying
$A\phi=\phi A$ must have three distinct principal curvatures.

\vskip 1mm

{\bf (ii) $(\dim V_\alpha,\dim V_\lambda)=(2,3)$ on $M$.}

In this case, we can take a local orthonormal frame field $\{X_{i}\}_{i=1}^{5}$
such that
$$
AX_i=\alpha X_i,\ i=1,5;\ \ \ \ AX_j=\lambda X_j,\ j=2,3,4,
$$
where $X_2=JX_1,X_4=JX_3,X_5=U$. Then by using \eqref{eqn:2.3}--\eqref{eqn:2.6} we get
\begin{equation}\label{eqn:3.1}
G(X_1,X_4)=G(X_2,X_3)=-JG(X_1,X_3),\ \ g(G(X_1,X_3),X_i)=0\ \ {\rm for}\ \ 1\leq i\leq 4,
\end{equation}
\begin{equation}\label{eqn:3.2}
g(G(X_1,X_3),G(X_1,X_3))=\tfrac{1}{3}.
\end{equation}

Let $\{e_{i}\}_{i=1}^{5}$ be the orthonormal basis as described in \eqref{eqn:2.22}. Then
$$
X_1=m e_1+n e_2+u e_3+v e_4,\ \ X_3=-u e_1+v e_2+m e_3-n e_4,
$$
for some functions $m,n,u,v$; and
$$
X_2=-n e_1+m e_2-v e_3+u e_4,\ \ X_4=-v e_1-u e_2+n e_3+m e_4.
$$

Now, taking in \eqref{eqn:2.19}, respectively, $(X,Y)=(X_1,X_3),\ (X_1,X_4),\ (X_2,X_3)$, $(X_2,X_4)$,
we can obtain
\begin{gather}
\tfrac{2}{3}c^2mv+\tfrac{2}{3}c^{2}nu=(\lambda-\alpha)g(G(X_1,\xi),X_3),\label{eqn:3.3}\\
-\tfrac{2}{3}c^2mu+\tfrac{2}{3}c^{2}nv=(\lambda-\alpha)g(G(X_1,\xi),X_4),\label{eqn:3.4}\\
-\tfrac{2}{3}c^2nv+\tfrac{2}{3}c^{2}mu=2(\lambda-\alpha)g(G(X_2,\xi),X_3),\label{eqn:3.5}\\
\tfrac{2}{3}c^2nu+\tfrac{2}{3}c^{2}mv=2(\lambda-\alpha)g(G(X_2,\xi),X_4).\label{eqn:3.6}
\end{gather}
From \eqref{eqn:3.4} and \eqref{eqn:3.5}, and respectively \eqref{eqn:3.3} and \eqref{eqn:3.6},
we deduce that
$$
g(G(X_1,X_3),U)=0,\ \ \ g(G(X_1,X_3),\xi)=0.
$$
This combining with \eqref{eqn:3.1} implies that $G(X_1,X_3)=0$, a contradiction to \eqref{eqn:3.2}.

\vskip 1mm

{\bf (iii) $(\dim V_\alpha,\dim V_\lambda)=(3,2)$ on $M$.}

In this case, as $U\in V_\alpha$, we have $\dim (V_\alpha\cap\{U\}^\perp)=\dim V_\lambda=2$.
For an orthonormal basis $\{X_1,X_2\}$ of $V_\alpha\cap\{U\}^\perp$ we consider $|g(JX_1,X_2)|$,
which is obviously independent of the choice of $\{X_1,X_2\}$, thus gives a well-defined function
$\theta:=|g(JX_1,X_2)|$ on $M$, with $0\leq\theta\leq1$. Since our concern is only local,
in order to prove that Case (iii) does not occur, we are sufficient to show that the
following three subcases do not occur on $M$.

\vskip 1mm

{\bf (iii)-(1)}. $0<\theta<1$.

In this subcase, we can take a local orthonormal frame field $\{X_{i}\}_{i=1}^{5}$ of $M$ such that
$$
AX_1=\alpha X_1,\ AX_2=\alpha X_2,\ AX_3=\lambda X_3,\ AX_4=\lambda X_4,\ X_5=U,
$$
where $X_3=(JX_1-\theta X_2)/\sqrt{1-\theta^2},\ X_4=(JX_2+\theta X_1)/\sqrt{1-\theta^2}$ and $\theta=g(JX_1,X_2)$.

Moreover, direct calculations give the following relations:
\begin{equation}\label{eqn:3.7}
\left\{
\begin{aligned}
&JX_1=\sqrt{1-\theta^2}X_3+\theta X_2,\ \ JX_2=\sqrt{1-\theta^2}X_4-\theta X_1,\\[-1mm]
&JX_3=-\sqrt{1-\theta^2}X_1-\theta X_4,\ \ JX_4=-\sqrt{1-\theta^2}X_2+\theta X_3,\\[-1mm]
&g(JX_1,X_2)=-g(JX_3,X_4)=\theta,\ \ g(JX_1,X_3)=g(JX_2,X_4)=\sqrt{1-\theta^2},\\[-1mm]
&g(JX_1,X_4)=g(JX_2,X_3)=0,\ \ G(X_3,X_4)=-G(X_1,X_2),\\
&G(X_1,X_3)=\tfrac{-\theta}{\sqrt{1-\theta^2}}G(X_1,X_2),\ \
G(X_1,X_4)=\tfrac{-1}{\sqrt{1-\theta^2}}JG(X_1,X_2),\\
&G(X_2,X_3)=\tfrac{1}{\sqrt{1-\theta^2}}JG(X_1,X_2),\ \
G(X_2,X_4)=\tfrac{-\theta}{\sqrt{1-\theta^2}}G(X_1,X_2).
\end{aligned}\right.
\end{equation}

Let $\{e_{i}\}_{i=1}^{5}$ be the orthonormal basis as described in \eqref{eqn:2.22}
and assume that
$$
X_{i}=\sum_{j=1} ^4 a_{ij}e_{j}, \ \ 1\le i\le4.
$$
Then, by the definition of $X_3$ and $X_4$, we can derive
\begin{equation}\label{eqn:3.8}
\left\{
\begin{aligned}
&a_{31}=\tfrac{-a_{12}-a_{21}\theta}{\sqrt{1-\theta^2}},\
a_{32}=\tfrac{a_{11}-a_{22}\theta}{\sqrt{1-\theta^2}},\
a_{33}=\tfrac{-a_{14}-a_{23}\theta}{\sqrt{1-\theta^2}},\
a_{34}=\tfrac{a_{13}-a_{24}\theta}{\sqrt{1-\theta^2}};\\
&
a_{41}=\tfrac{-a_{22}+a_{11}\theta}{\sqrt{1-\theta^2}},\
a_{42}=\tfrac{a_{21}+a_{12}\theta}{\sqrt{1-\theta^2}},\
a_{43}=\tfrac{-a_{24}+a_{13}\theta}{\sqrt{1-\theta^2}},\
a_{44}=\tfrac{a_{23}+a_{14}\theta}{\sqrt{1-\theta^2}}.
\end{aligned}
\right.
\end{equation}

Taking, in \eqref{eqn:2.19}, $(X,Y)=(X_i,X_j)$ for $1\leq i<j\leq4$, and
using \eqref{eqn:3.7} and \eqref{eqn:2.22}, we get
\begin{equation}\label{eqn:3.9}
-\tfrac{1}{6}\theta+\tfrac{2}{3}c^{2}(a_{11}a_{22}-a_{12}a_{21})=0,
\end{equation}
\begin{equation}\label{eqn:3.10}
\tfrac{2}{3\sqrt{1-\theta^2}}c^2(a_{11}a_{21}+a_{12}a_{22})
+\tfrac{1}{\sqrt{1-\theta^2}}(\alpha-\lambda)g(G(X_1,X_2),U)=0,
\end{equation}
\begin{equation}\label{eqn:3.11}
\tfrac{2}{3\sqrt{1-\theta^2}}c^2(a_{11}a_{21}+a_{12}a_{22})
-\tfrac{1}{\sqrt{1-\theta^2}}(\alpha-\lambda)g(G(X_1,X_2),U)=0,
\end{equation}
\begin{equation}\label{eqn:3.12}
\begin{aligned}
\tfrac{2}{3\sqrt{1-\theta^2}}c^2(&-a_{11}^2-a_{12}^2+(a_{22}a_{11}-a_{21}a_{12})\theta)+\tfrac{\sqrt{1-\theta^2}}{6}\\[-1mm]
&-\tfrac{\theta}{\sqrt{1-\theta^2}}(\alpha-\lambda)g(G(X_1,X_2),\xi)+\alpha(\alpha-\lambda)\sqrt{1-\theta^2}=0,
\end{aligned}
\end{equation}
\begin{equation}\label{eqn:3.13}
\begin{aligned}
\tfrac{2}{3\sqrt{1-\theta^2}}c^2(&-a_{21}^2-a_{22}^2+(a_{22}a_{11}-a_{21}a_{12})\theta)+\tfrac{\sqrt{1-\theta^2}}{6}\\[-1mm]
&-\tfrac{\theta}{\sqrt{1-\theta^2}}(\alpha-\lambda)g(G(X_1,X_2),\xi)
+\alpha(\alpha-\lambda)\sqrt{1-\theta^2}=0,
\end{aligned}
\end{equation}
\begin{equation}\label{eqn:3.14}
\begin{aligned}
\tfrac{2}{3(1-\theta^2)}c^2\big[&a_{21}a_{12}-a_{11}a_{22}+(a_{22}^2+a_{11}^2+a_{21}^2+a_{12}^2)\theta
+(a_{12}a_{21}-a_{11}a_{22})\theta^2\big]\\[-1mm]
&-\tfrac{\theta}{6}-2(\alpha-\lambda)g(G(X_1,X_2),\xi)-2\lambda(\alpha-\lambda)\theta=0.
\end{aligned}
\end{equation}

From \eqref{eqn:3.10}, \eqref{eqn:3.11} and $g(X_1,X_2)=0$, we have
$$
g(G(X_1,X_2),U)=0,\ \ a_{11}a_{21}+a_{12}a_{22}=0,\ \ a_{13}a_{23}+a_{14}a_{24}=0.
$$

From \eqref{eqn:3.9}, \eqref{eqn:3.12}, \eqref{eqn:3.13} and $g(X_1,X_1)=g(X_2,X_2)=1$, we have
$$
a_{11}^2+a_{12}^2=a_{21}^2+a_{22}^2\not=0,\ \ a_{13}^2+a_{14}^2=a_{23}^2+a_{24}^2.
$$
Thus, we can write
$$
\left\{
\begin{aligned}
&a_{11}=\sqrt{a_{11}^2+a_{12}^2}\cos \omega_{1},\ \ a_{12}=\sqrt{a_{11}^2+a_{12}^2}\sin \omega_{1};\\
&a_{21}=\sqrt{a_{11}^2+a_{12}^2}\cos \omega_{2},\ \ a_{22}=\sqrt{a_{11}^2+a_{12}^2}\sin\omega_{2}.
\end{aligned}
\right.
$$
Then the fact $0=a_{11}a_{21}+a_{12}a_{22}=(a_{11}^2+a_{12}^2)\cos(\omega_{1}-\omega_{2})$
implies that $\omega_{1}-\omega_{2}=\tfrac{\pi}{2}(2k+1)$ for $k\in\mathbb Z$. Hence
$(a_{21},a_{22})=\pm\,(a_{12},-a_{11})$. On the other hand, \eqref{eqn:3.9} implies that
$a_{11}a_{22}-a_{12}a_{21}=\tfrac{\theta}{4c^2}>0$, so it should be that $(a_{21},a_{22})=-(a_{12},-a_{11})$.

Similarly, we can prove that $(a_{23},a_{24})=(a_{14},-a_{13})$.
It follows that $a_{11}^2+a_{12}^2=\tfrac{\theta}{4c^2}$ and $a_{13}^2+a_{14}^2=1-\tfrac{\theta}{4c^2}$.
On the other hand, by definition, we can finally get
$$
\theta=\sum a_{1i}a_{2j}g(Je_i,e_j)=a_{11}a_{22}-a_{12}a_{21}+a_{13}a_{24}-a_{14}a_{23}
=\tfrac{\theta}{2c^2}-1,
$$
and thus $\theta=\tfrac{2c^2}{1-2c^2}$.

Next, from the fact $g(G(X_1,X_2),X_i)=0$ for $1\leq i \leq 5$ and that, by \eqref{eqn:2.6},
$$
g(G(X_1,X_2),G(X_1,X_2))=\tfrac13(1-\theta^2),
$$
we have $G(X_1,X_2)=\pm\sqrt{(1-\theta^2)/3}\,\xi$. Since the discussion is totally
similar, we just consider the case $G(X_1,X_2)=\sqrt{(1-\theta^2)/3}\,\xi$.
We calculate the connections $\{\nabla_{X_{i}}X_{j}\}$ so that we can apply
for the Codazzi equations.

Put $\nabla_{X_i}X_j=\sum \Gamma_{ij}^{k}X_{k}$ with $\Gamma_{ij}^{k}=-\Gamma_{ik}^{j}$, $1\le i,j,k\le 5$.

Then, on the one hand, by definition and the Gauss-Weingarten formulas, we have
\begin{equation*}\label{eq:cs1}
G(X_1,\xi)=-\sum_{i=1}^5 \Gamma_{15}^{i}X_i+\alpha JX_1.
\end{equation*}
On the other hand, using $G(X_1,\xi)=\sum_ig(G(X_1,\xi),X_i)X_i$, we easily get
$$
G(X_1,\xi)=-\sqrt{\tfrac{1-\theta^2}{3}}X_2+\tfrac{\sqrt{3}}{3}\theta X_3.
$$

From the above calculations and \eqref{eqn:3.7}, it follows that
\begin{equation}\label{eqn:3.15}
\Gamma_{15}^{1}=0,\ \ \Gamma_{15}^{2}=\alpha\theta+\sqrt{\tfrac{1-\theta^2}{3}},\ \
\Gamma_{15}^{3}=\alpha\sqrt{1-\theta^2}-\tfrac{\sqrt{3}}{3}\theta,\ \ \Gamma_{15}^{4}=0.
\end{equation}

Analogously, calculating $G(X_i,\xi)=(\tilde \nabla_{X_i} J)\xi$ for $2\leq i\leq 4$,
we can further obtain
\begin{equation}\label{eqn:3.16}
\left\{
\begin{aligned}
&\Gamma_{25}^{1}=-\alpha\theta-\sqrt{\tfrac{1-\theta^2}{3}},\ \ \Gamma_{25}^{2}=0,\ \ \Gamma_{25}^{3}=0,\ \ \Gamma_{25}^{4}=\alpha\sqrt{1-\theta^2}-\tfrac{\sqrt{3}}{3}\theta,\\
&\Gamma_{35}^{1}=-\lambda\sqrt{1-\theta^2}+\tfrac{\sqrt{3}}{3}\theta,\ \ \Gamma_{35}^{2}=0,\ \ \Gamma_{35}^{3}=0,\ \ \Gamma_{35}^{4}=-\lambda\theta-\sqrt{\tfrac{1-\theta^2}{3}},\\
&\Gamma_{45}^{1}=0,\ \ \Gamma_{45}^{2}=-\lambda\sqrt{1-\theta^2}+\tfrac{\sqrt{3}}{3}\theta,\ \ \Gamma_{45}^{3}=\lambda\theta+\sqrt{\tfrac{1-\theta^2}{3}},\ \ \Gamma_{45}^{4}=0.
\end{aligned}\right.
\end{equation}

Now, we are ready to calculate $(\nabla_U A){e_i}-(\nabla_{e_i} A)U$ for $1\leq i\leq 4$.

On the one hand, using $e_{i}=\sum_{j=1}^4{a_{ji}}X_{j}$ and the preceding results
\eqref{eqn:3.15} and \eqref{eqn:3.16}, direct calculations give the ${\{U\}^\bot}$-components of
$(\nabla_U A){e_i}-(\nabla_{e_i} A)U$:
$$
\left(
\begin{aligned}
&(\nabla_UA)e_1-(\nabla_{e_1}A)U\\
&(\nabla_UA)e_2-(\nabla_{e_2}A)U\\
&(\nabla_UA)e_3-(\nabla_{e_3}A)U\\
&(\nabla_UA)e_4-(\nabla_{e_4}A)U
\end{aligned}
\right)_{{\{U\}^\bot}}=BC
\left(
\begin{aligned}
&X_1\\
&X_2\\
&X_3\\
&X_4
\end{aligned}
\right),
$$
where
$$
B=(a_{ij})^T=
\left(
  \begin{array}{cccc}
    a_{11} & -a_{12} & -a_{12}\sqrt{(1-\theta)/(1+\theta)} & -a_{11}\sqrt{(1-\theta)/(1+\theta)}\\[1mm]
    a_{12} & a_{11} & a_{11}\sqrt{(1-\theta)/(1+\theta)} & -a_{12}\sqrt{(1-\theta)/(1+\theta)}\\[1mm]
    a_{13} & a_{14} & -a_{14}\sqrt{(1+\theta)/(1-\theta)} & a_{13}\sqrt{(1+\theta)/(1-\theta)}\\[1mm]
    a_{14} & -a_{13} & a_{13}\sqrt{(1+\theta)/(1-\theta)} & a_{14}\sqrt{(1+\theta)/(1-\theta)}
  \end{array}
\right),
$$
$$
C=(C_{ij}):=
\left(
  \begin{array}{cccc}
    U(\alpha) & 0 & (\alpha-\lambda)(\Gamma_{51}^3-\Gamma_{15}^3) & (\alpha-\lambda)\Gamma_{51}^4 \\
    0 & U(\alpha) & (\alpha-\lambda)\Gamma_{52}^3 & (\alpha-\lambda)(\Gamma_{52}^4-\Gamma_{25}^4) \\
    (\lambda-\alpha)\Gamma_{53}^1 & (\lambda-\alpha)\Gamma_{53}^2 & U(\lambda) & (\lambda-\alpha)\Gamma_{35}^4 \\
    (\lambda-\alpha)\Gamma_{54}^1 & (\lambda-\alpha)\Gamma_{54}^2 & (\lambda-\alpha)\Gamma_{45}^3 & U(\lambda)
  \end{array}
\right).
$$

On the other hand, using the Codazzi equation \eqref{eqn:2.18}, $e_{i}=\sum_{j=1}^4{a_{ji}}X_{j}$
and \eqref{eqn:2.22}, another calculation for the ${\{U\}^\bot}$-components of
$(\nabla_U A)e_{i}-(\nabla_{e_{i}} A)U$ can be carried out to obtain:
$$
\left(
\begin{aligned}
&(\nabla_UA)e_1-(\nabla_{e_1}A)U\\
&(\nabla_UA)e_2-(\nabla_{e_2}A)U\\
&(\nabla_UA)e_3-(\nabla_{e_3}A)U\\
&(\nabla_UA)e_4-(\nabla_{e_4}A)U
\end{aligned}
\right)_{{\{U\}^\bot}}=(D+E)
\left(
\begin{aligned}
&X_1\\
&X_2\\
&X_3\\
&X_4
\end{aligned}
\right),
$$
where
$$
D=(D_{ij}):=
\left(
  \begin{array}{cccc}
    -\tfrac{2ab}{3}a_{11} & \tfrac{2ab}{3}a_{12} & \tfrac{2aba_{12}}{3}\sqrt{\tfrac{1-\theta}{1+\theta}} & \tfrac{2aba_{11}}{3}\sqrt{\tfrac{1-\theta}{1+\theta}} \\[1mm]
    \tfrac{2ab}{3}a_{12} & \tfrac{2ab}{3}a_{11} & \tfrac{2aba_{11}}{3}\sqrt{\tfrac{1-\theta}{1+\theta}} & -\tfrac{2aba_{12}}{3}\sqrt{\tfrac{1-\theta}{1+\theta}} \\[1mm]
    \tfrac{b}{3}a_{13} & \tfrac{b}{3}a_{14} & -\tfrac{ba_{14}}{3}\sqrt{\tfrac{1+\theta}{1-\theta}} & \tfrac{ba_{13}}{3}\sqrt{\tfrac{1+\theta}{1-\theta}} \\[1mm]
    -\tfrac{b}{3}a_{14} & \tfrac{b}{3}a_{13} & -\tfrac{ba_{13}}{3}\sqrt{\tfrac{1+\theta}{1-\theta}} & -\tfrac{ba_{14}}{3}\sqrt{\tfrac{1+\theta}{1-\theta}}
  \end{array}
\right),
$$
$$
E=(E_{ij}):=
\left(
  \begin{array}{cccc}
     \tfrac{8a^2-3}{12}a_{12} & \tfrac{8a^2-3}{12}a_{11} & \tfrac{(8a^2-3)a_{11}}{12}\sqrt{\tfrac{1-\theta}{1+\theta}} & -\tfrac{(8a^2-3)a_{12}}{12}\sqrt{\tfrac{1-\theta}{1+\theta}}\\
     \tfrac{3-8b^2}{12}a_{11} & \tfrac{8b^2-3}{12}a_{12} & \tfrac{(8b^2-3)a_{12}}{12}\sqrt{\tfrac{1-\theta}{1+\theta}} & \tfrac{(8b^2-3)a_{11}}{12}\sqrt{\tfrac{1-\theta}{1+\theta}}\\
     \tfrac{1-4a}{12}a_{14} & -\tfrac{1-4a}{12}a_{13} & \tfrac{(1-4a)a_{13}}{12}\sqrt{\tfrac{1+\theta}{1-\theta}} &  \tfrac{(1-4a)a_{14}}{12}\sqrt{\tfrac{1+\theta}{1-\theta}}\\
     -\tfrac{1+4a}{12}a_{13} & -\tfrac{1+4a}{12}a_{14} & \tfrac{(1+4a)a_{14}}{12}\sqrt{\tfrac{1+\theta}{1-\theta}} & -\tfrac{(1+4a)a_{13}}{12}\sqrt{\tfrac{1+\theta}{1-\theta}}
  \end{array}
\right).
$$

In this way, we obtain the equation $BC=D+E$. This can be written in equivalent form:
$C_{ij}=\sum_ka_{ik}(D_{kj}+E_{kj})$ for $1\le i,j\le4$. Then, since by \eqref{eqn:3.16} we have
$$
C_{11}-C_{22}=0, \ C_{12}+C_{21}=0,\ C_{33}-C_{44}=0,\ C_{34}+C_{43}=0,
$$
it follows that $LF=0$, where $L=(a_{11}^2-a_{12}^2, \, a_{13}^2-a_{14}^2, \, a_{11}a_{12},\, a_{13}a_{14})$, and
$$
F=\left(
  \begin{array}{cccc}
    -2ab & a^2-b^2 & (b^2-a^2)(1-\theta)^2 & 2ab(1-\theta)^2\\
    b & a & -a(1+\theta)^2 & -b(1+\theta)^2\\
    2(a^2-b^2) & 4ab & -4ab(1-\theta)^2 & 2(b^2-a^2)(1-\theta)^2\\
    -2a & 2b & -2b(1+\theta)^2 & 2a(1+\theta)^2
  \end{array}\right).
$$
%$$
%F=\left(
%  \begin{array}{cccc}
%    \tfrac{-2ab}{3} & \tfrac{2a^2-2b^2}{3} & \tfrac{-2a^2+2b^2}{3}(1-\theta)^2 & \tfrac{2ab}{3}(1-\theta)^2\\
%    \tfrac{b}{3} & \tfrac{2a}{3} & \tfrac{-2a}{3}(1+\theta)^2 & \tfrac{-b}{3}(1+\theta)^2\\
%    \tfrac{2a^2-2b^2}{3} & \tfrac{8ab}{3} & \tfrac{-8ab}{3}(1-\theta)^2 & \tfrac{-2a^2+2b^2}{3}(1-\theta)^2\\
%    \tfrac{-2a}{3} & \tfrac{4b}{3} & \tfrac{-4b}{3}(1+\theta)^2 & \tfrac{2a}{3}(1+\theta)^2
%  \end{array}\right).
%$$

Now, direct calculation gives that $\det F=-64\theta^2(a^2+b^2)^3$.

If $\det F=0$, then $c=1$ and this contradicts to $\theta=\tfrac{2c^2}{1-2c^2}\in (0,1)$.
If $\det F\neq0$, then $L=0$ and thus $a_{11}=a_{12}=a_{13}=a_{14}=0$, which is also a contradiction.

In summary, we have shown that {\bf (iii)-(1)} does not occur.

\vskip 1mm

{\bf (iii)-(2)}. $\theta=0$.

In this case, we have $J\{V_\alpha\cap\{U\}^\perp\}=V_\lambda$. Take a local orthonormal frame field
$\{X_{i}\}_{i=1}^{5}$ of $M$ such that
$$
AX_1=\alpha X_1,\ AX_2=\alpha X_2,\ AX_3=\lambda X_3,\ AX_4=\lambda X_4,\ AX_5=\alpha X_5,
$$
where $X_3=JX_1,\ X_4=JX_2,\ X_5=U$. It follows that
$$
g(G(X_1,X_2),X_i)=0,\ 1\leq i\le 4;\ \ \ g(G(X_1,X_2),G(X_1,X_2))=\tfrac{1}{3}.
$$
Assume that $X_{i}=\sum_{j=1} ^4 a_{ij}e_{j}$ for $1\leq i\leq 4$. Then taking in \eqref{eqn:2.19}
that $(X,Y)=(X_i,X_j)$ for each $1\le i,j\le4$, we can still get the equations from \eqref{eqn:3.9}
up to \eqref{eqn:3.14} but with $\theta=0$.
From \eqref{eqn:3.9} and \eqref{eqn:3.14} corresponding to $\theta=0$, we get $g(G(X_1,X_2),\xi)=0$.
Then, by \eqref{eqn:3.10} and \eqref{eqn:3.11}, we obtain $g(G(X_1,X_2),U)=0$.

It follows that $G(X_1,X_2)=0$, a contradiction to $g(G(X_1,X_2),G(X_1,X_2))=\tfrac{1}{3}$.

\vskip 1mm

{\bf (iii)-(3)}. $\theta=1$.

In this case, both $V_\alpha\cap\{U\}^\perp$ and $V_\lambda$ are $J$-invariant. Then, it is easily
seen that $M$ satisfies $A\phi=\phi A$, and according to Theorem 4.1 of \cite{H-Y-Z} once more we
get as desired a contradiction.

\vskip 1mm

{\bf (iv) $(\dim V_\alpha,\dim V_\lambda)=(4,1)$ on $M$.}

In this case, we can take a local orthonormal basis $\{X_{i}\}_{i=1}^{5}$ such that
$$
AX_1=\lambda X_1,\ AX_2=\alpha X_2,\ AX_3=\alpha X_3,\ AX_4=\alpha X_4,\ AX_5=\alpha X_5,
$$
where $X_2=JX_1,\ X_4=JX_3,\ X_5=U$. Then as preceding we have
\begin{equation}\label{eqn:3.17}
g(G(X_1,X_3),X_i)=0,\ 1\leq i\leq4;\ \ |G(X_1,X_3)|^2=\tfrac{1}{3}.
\end{equation}

Let $\{e_{i}\}_{i=1}^{5}$ be the orthonormal basis as described in \eqref{eqn:2.22}
and assume, for some functions $m,n,u,v$ that $X_1=m e_1+n e_2+ue_3+ve_4,\ \ X_3=-ue_1+ve_2+m e_3-n e_4$.
Then, by definition, we have
$$
X_2=-n e_1+m e_2-ve_3+ue_4,\ \ X_4=-ve_1-ue_2+n e_3+m e_4.
$$
Taking in \eqref{eqn:2.19}, respectively, $(X,Y)=(X_1,X_3),(X_1,X_4),(X_3,X_2),(X_4,X_2)$, we get
\begin{gather}
\tfrac{2}{3}c^2mv+\tfrac{2}{3}c^{2}nu=(\lambda-\alpha)g(G(X_1,\xi),X_3),\label{eqn:3.18}\\
-\tfrac{2}{3}c^2mu+\tfrac{2}{3}c^{2}nv=(\lambda-\alpha)g(G(X_1,\xi),X_4),\label{eqn:3.19}\\
-\tfrac{2}{3}c^2mu+\tfrac{2}{3}c^{2}nv=0,\label{eqn:3.20}\\
\tfrac{2}{3}c^2nu+\tfrac{2}{3}c^{2}mv=0.\label{eqn:3.21}
\end{gather}
From these equations we immediately obtain
$$
g(G(X_1,X_3),U)=0,\ \ \ g(G(X_1,X_3),\xi)=0.
$$
This together with \eqref{eqn:3.17} gives $G(X_1,X_3)=0$, a contradiction to $|G(X_1,X_3)|^2=\tfrac{1}{3}$.

This finally completes the proof of Lemma \ref{lemma:3.1}.
\end{proof}

\begin{lemma}\label{lemma:3.2}
The case $\dim\mathfrak{D}=2$ does not occur.
\end{lemma}
\begin{proof}
Suppose on the contrary that $\dim\mathfrak{D}=2$ does hold on $M$.

Then, we consider each possibility of the dimensions $(\dim V_\alpha,\dim V_\lambda)$.

\vskip 1mm

{\bf (i) $(\dim V_\alpha,\dim V_\lambda)=(1,4)$ on $M$.}

In this case, we can easily show that $M$ satisfies $A\phi=\phi A$. As before
by Theorem 4.1 in \cite{H-Y-Z} this is impossible.

\vskip 1mm

{\bf (ii) $(\dim V_\alpha,\dim V_\lambda)=(2,3)$ on $M$.}

In this case, we take a local orthonormal frame field $\{X_{i}\}_{i=1}^{5}$ of $M$ such that
$$
AX_1=\alpha X_1,\ AX_2=\lambda X_2,\ AX_3=\lambda X_3,\ AX_4=\lambda X_4,\ AX_5=\alpha X_5,
$$
where $X_2=JX_1,\ X_4=JX_3,\ X_5=U$. By \eqref{eqn:2.3}--\eqref{eqn:2.5}, $G(X_1,\xi)$ is orthogonal to ${\rm Span}\{\xi,U,X_1,X_2\}$,
so $AG(X_1,\xi)=\lambda G(X_1,\xi)$.
Then, taking $X=X_1$ in \eqref{eqn:2.19}, we can get
\begin{equation}\label{eqn:3.22}
(\alpha-\lambda)g(G(X_1,\xi),Y)=(\alpha^2-\alpha\lambda+\tfrac{1}{6})g(X_2,Y),\ \forall\, Y\in\{U\}^\perp.
\end{equation}

Notice that $g(X_2,X_3)=g(X_2,X_4)=0$ and $\alpha\neq\lambda$, so \eqref{eqn:3.22} implies that $G(X_1,\xi)=0$.
However, by \eqref{eqn:2.6} we have $|G(X_1,\xi)|^2=\tfrac{1}{3}$. This is a contradiction.

\vskip 1mm

{\bf (iii) $(\dim V_\alpha,\dim V_\lambda)=(3,2)$ on $M$.}

In this case, we take a local orthonormal frame field $\{X_{i}\}_{i=1}^{5}$ of $M$ such that
$$
AX_1=\alpha X_1,\ AX_2=\alpha X_2,\ AX_3=\lambda X_3,\ AX_4=\lambda X_4,\ AX_5=\alpha X_5,
$$
where $X_5=U$. Taking in \eqref{eqn:2.19} $(X,Y)=(X_1,X_2)$ gives $g(\phi X_1,X_2)=0$. It follows
that $J\{V_\alpha\cap\{U\}^\perp\}=V_\lambda$. Then, we can choose a local orthonormal frame field
$\{\tilde{X}_i\}_{i=1}^{5}$ such that $\tilde{X}_1=X_1,\ \tilde{X}_2=J\tilde{X}_1,\ \tilde{X}_3=X_2,
\ \tilde{X}_4=J\tilde{X}_3,\ \tilde{X}_5=U$, and moreover, $\tilde{X}_1,\tilde{X}_3,\tilde{X}_5\in V_\alpha$
and $\tilde{X}_2,\tilde{X}_4\in V_\lambda$. By the identity \eqref{eqn:2.19} with $(X,Y)$ equal to $(\tilde{X}_2,\tilde{X}_3),\ (\tilde{X}_2,\tilde{X}_4)$, respectively, we have $g(G(\tilde{X}_2,\xi),\tilde{X}_3)=g(G(\tilde{X}_2,\xi),\tilde{X}_4)=0$.
This implies that $G(\tilde{X}_2,\xi)=0$ due to the obvious fact $G(\tilde{X}_2,\xi)\perp{\rm Span}\,\{\xi,U,\tilde{X}_1,\tilde{X}_2\}$.

However, by \eqref{eqn:2.6} we have $|G(\tilde{X}_2,\xi)|^2=\tfrac{1}{3}$. This is also a contradiction.

\vskip 1mm

{\bf (iv) $(\dim V_\alpha,\dim V_\lambda)=(4,1)$ on $M$.}

In this case, we take a local orthonormal frame field $\{X_{i}\}_{i=1}^{5}$ of $M$ such that
$$
AX_1=\lambda X_1,\ AX_2=\alpha X_2,\ AX_3=\alpha X_3,\ AX_4=\alpha X_4,\ AX_5=\alpha X_5,
$$
where $X_2=JX_1,\ X_4=JX_3,\ X_5=U$. By \eqref{eqn:2.3}--\eqref{eqn:2.5}, $G(X_1,\xi)$ is orthogonal to ${\rm Span}\{\xi,U,X_1,X_2\}$,
so $AG(X_1,\xi)=\alpha G(X_1,\xi)$.
Taking in \eqref{eqn:2.19} $X=X_1$, we get
\begin{equation}\label{eqn:3.23}
(\alpha-\lambda)g(G(X_1,\xi),Y)=(\alpha^2-\alpha\lambda+\tfrac{1}{6})g(X_2,Y),\ \forall\, Y\in\{U\}^\perp.
\end{equation}

Then, similar as in case {\bf (ii)}, from \eqref{eqn:3.23}, the fact $g(X_2,X_3)=g(X_2,X_4)=0$
and $\alpha\neq\lambda$, we obtain $G(X_1,\xi)=0$.

However, by \eqref{eqn:2.6}, $|G(X_1,\xi)|^2=\tfrac{1}{3}$. This is a contradiction.
\end{proof}

%==================================================
\section{Examples of Hopf hypersurfaces in $\mathbf{S}^3\times\mathbf{S}^3$}\label{sect:4}
%%%%%%%%%%%%%%%%%%%%%%%%%%%%%%%%%%%%%%%%%%%%%%%%%%%
As usual we denote $\mathbf{S}^3$ (resp. $\mathbf{S}^2$) the set of the unitary (resp. imaginary)
quaternions in the quaternion space $\mathbb{H}$. Then, in this short section, we can describe
several of the simplest examples of Hopf hypersurfaces in the NK $\mathbf{S}^3\times\mathbf{S}^3$.
\begin{examples}\label{exam:4.1}
For each $0<r\leq1$, we define three families of hypersurfaces
$M^{(r)}_1$, $M^{(r)}_2$ and $M^{(r)}_3$ in the NK $\mathbf{S}^3\times\mathbf{S}^3$ as below:
\begin{equation*}
\begin{aligned}
&M^{(r)}_1:=\left\{(x,\sqrt{1-r^2}+ry)\in\mathbf{S}^3\times
\mathbf{S}^3 ~|~x\in \mathbf{S}^3,\ y\in\mathbf{S}^2\right\},\\
&M^{(r)}_2:=\mathcal{F}_1(M^{(r)}_1),\\%\qquad \qquad
&M^{(r)}_3:=\mathcal{F}_2(M^{(r)}_1).
\end{aligned}
\end{equation*}
\end{examples}

\begin{remark}\label{rem:4.1}
Among the preceding hypersurfaces $M^{(r)}_1$, $M^{(r)}_2$ and $M^{(r)}_3$ of the NK
$\mathbf{S}^3\times\mathbf{S}^3$, $M^{(r)}_1$, $M^{(r)}_2$ and $M^{(1)}_3$ have been
carefully discussed, respectively, in Examples 5.1, 5.2 and 5.3 of \cite{H-Y-Z}.
As a matter of fact, all of them are Hopf hypersurfaces with three distinct
constant principal curvatures: $\alpha=0$ (i.e. $AU=0$) of multiplicity $1$,
$\lambda=\tfrac{\sqrt{1-r^2}}{2r}-\tfrac{\sqrt{3-2r^2}}{2\sqrt{3}r}$ of multiplicity $2$,
and $\beta=\tfrac{\sqrt{1-r^2}}{2r}+\tfrac{\sqrt{3-2r^2}}{2\sqrt{3}r}$ of multiplicity $2$.
The holomorphic distributions $\{U\}^\perp$ of these hypersurfaces are all preserved
by the almost product structure $P$ of the NK $\mathbf{S}^3\times\mathbf{S}^3$, but
$P$ acts differently on their unit normal vector fields.
\end{remark}

\begin{examples}\label{exam:4.2}
For each $0<k,l<1$, $k^2+l^2=1$, we can define three families of hypersurfaces
$M^{(k,l)}_4$, $M^{(k,l)}_5$ and $M^{(k,l)}_6$ in the NK $\mathbf{S}^3\times\mathbf{S}^3$ as below:
\begin{equation*}
\begin{aligned}
&M^{(k,l)}_4:=\left\{(x,(y_1,y_2,y_3,y_4))\in \mathbf{S}^3 \times
\mathbf{S}^3 ~|~x\in \mathbf{S}^3,\ y_1^2+y_2^2=k^2,\ y_3^2+y_4^2=l^2\right\},\\
&M^{(k,l)}_5:=\mathcal{F}_1(M^{(k,l)}_4),\\%\qquad \qquad
&M^{(k,l)}_6:=\mathcal{F}_2(M^{(k,l)}_4).
\end{aligned}
\end{equation*}
\end{examples}

\begin{remark}\label{rem:4.2}
Direct calculations show that all of these three families of hypersurfaces are
Hopf ones, and they have five distinct constant principal curvatures: $\alpha=0$
(i.e. $AU=0$), $\lambda_1=\tfrac{3k-\sqrt{9k^2+3l^2}}{6l}$,
$\lambda_2=\tfrac{3k+\sqrt{9k^2+3l^2}}{6l}$, $\lambda_3=\tfrac{-3l-\sqrt{3k^2+9l^2}}{6k}$,
$\lambda_4=\tfrac{-3l+\sqrt{3k^2+9l^2}}{6k}$. Similarly, the holomorphic distributions
$\{U\}^\perp$ of these hypersurfaces are all preserved by the almost product structure
$P$ of the NK $\mathbf{S}^3\times\mathbf{S}^3$, but $P$ acts differently on their unit
normal vector fields.
\end{remark}

\begin{remark}\label{rem:4.3}
Theorem \ref{thm:1.2} gives a characterization of the Hopf hypersurfaces $M^{(r)}_1$,
$M^{(r)}_2$ and $M^{(r)}_3$ in the NK $\mathbf{S}^3\times\mathbf{S}^3$.
We expect that a similar interesting characterization of the Hopf hypersurfaces
$M^{(k,l)}_4$, $M^{(k,l)}_5$ and $M^{(k,l)}_6$ in the NK $\mathbf{S}^3\times\mathbf{S}^3$
is possible, but at the moment it is still not achieved.
\end{remark}

%==================================================
\section{The proof of Theorem \ref{thm:1.2}}\label{sect:5}
%%%%%%%%%%%%%%%%%%%%%%%%%%%%%%%%%%%%%%%%%%%%%%%%%%%

This last section is devoted to the proof of Theorem \ref{thm:1.2}, which is
given in two steps. In the sequel, we assume that $M$ is a Hopf hypersurface
of the NK $\mathbf{S}^3\times\mathbf{S}^3$ with three distinct principal
curvatures $\alpha,\ \lambda$ and $\beta$ such that $AU=\alpha U$, and that
$P\{U\}^\perp=\{U\}^\perp$. In particular, \eqref{eqn:2.23} holds.

\subsection{The principal curvatures and their multiplicities}\label{sect:5.1}~

Let $V_\alpha,V_\lambda$ and $V_\beta$ denote the eigenspaces corresponding
to the principal curvatures $\alpha,\lambda$ and $\beta$, respectively.
By the assumption of having three distinct principal curvatures and the
continuity of the principal curvature functions, we know that the
dimensions $(\dim V_\alpha,\dim V_\lambda,\dim V_\beta)$ remain unchanged
on $M$, which, without loss of generality, have four possibilities:
$(3,1,1)$, $(2,2,1)$, $(1,3,1)$ and $(1,2,2)$.

First of all, we shall determine the multiplicities of the principal curvatures.

\begin{lemma}\label{lemma:5.1}
The multiplicities of the three distinct principal curvature functions $\alpha,\ \lambda,\ \beta$ can
only be $1,2$ and $2$, respectively.
\end{lemma}
\begin{proof}
Suppose on the contrary that, for the multiplicities of the principal curvatures $\alpha,\lambda$ and
$\beta$, one of the three possibilities $(3,1,1),(2,2,1),(1,3,1)$ does occur. Then, for each possible
case, we shall derive a contradiction by using Lemma \ref{lemma:2.1}.

\vskip 1mm

{\bf (i) $(\dim V_\alpha,\dim V_\lambda,\dim V_\beta)=(3,1,1)$ on $M$.}

We take a local orthonormal frame field $\{X_{i}\}_{i=1}^{5}$ of $M$ such that
$$
AX_1=\lambda X_1,\ AX_2=\beta X_2,\ AX_3=\alpha X_3,\ AX_4=\alpha X_4,\ X_5=U.
$$

Taking in \eqref{eqn:2.19} $(X,Y)=(X_3,X_4)$, we get $g(\phi X_3,X_4)=0$, which implies that
$J\{V_\lambda\oplus V_\beta\}=V_\alpha\cap{\{U\}^{\perp}}$. So we can further choose
$X_3=JX_1$ and $X_4=JX_2$. Then, we easily show that $G(X_1,X_2)\in {\rm Span}\{\xi,U\}$, and
by \eqref{eqn:2.6}, we have $|G(X_1,X_2)|^2=\tfrac{1}{3}$.

Now, taking in \eqref{eqn:2.19} $(X,Y)=(X_1,X_3),\ (X_2,X_4),\ (X_2,X_3),\ (X_1,X_2)$, respectively,
we obtain
\begin{gather}
\alpha^2-\alpha\lambda=-\tfrac{1}{6},\ \ \alpha^2-\alpha\beta=-\tfrac{1}{6},\label{eqn:5.1}\\
(\alpha-\beta)g(G(X_1,X_2),U)=0, \ \ (2\alpha-\lambda-\beta)g(G(X_1,X_2),\xi)=0.\label{eqn:5.2}
\end{gather}

From \eqref{eqn:5.2}, $\alpha-\beta\neq0$ and the preceding results, we see that $g(G(X_1,X_2),\xi)\not=0$
and $\lambda+\beta=2\alpha$. On the other hand, from \eqref{eqn:5.1} we get $2\alpha^2-\alpha(\lambda+
\beta)=-\tfrac{1}{3}$. But this is a contradiction to $\lambda+\beta=2\alpha$.

\vskip 1mm

{\bf (ii) $(\dim V_\alpha,\dim V_\lambda,\dim V_\beta)=(2,2,1)$ on $M$.}

In this case, we can define a function $\theta:=|g(JX,Y)|$ on $M$ for unit vectors
$X\in V_\alpha\cap\{U\}^\perp$ and $Y\in V_\beta$. Since $0\leq\theta\leq1$ and that
our concern is only local, in order to prove that Case (ii) does not occur, it is
sufficient to show that the following three subcases do not occur on $M$.

\vskip 1mm

{\bf (ii)-(a)}. $0<\theta<1$.

In this subcase, we have the decomposition $JX=W+g(JX,Y)Y$ and $0\not=W\in V_\lambda$. Then,
we can take a local orthonormal frame field $\{X_{i}\}_{i=1}^{5}$ of $M$ such that
$$
AX_1=\alpha X_1,\ AX_2=\beta X_2,\ AX_3=\lambda X_3,\ AX_4=\lambda X_4,\ X_5=U,
$$
where $X_3=(JX_1-\theta X_2)/\sqrt{1-\theta^2},\ X_4=(JX_2+\theta X_1)/\sqrt{1-\theta^2}$
and $\theta=g(JX_1,X_2)$.

It follows that $G(X_1,X_2)\in {\rm Span}\{\xi,U\}$ and, by \eqref{eqn:2.6}, $|G(X_1,X_2)|^2=(1-\theta^2)/3$.
Moreover, it is easily seen that with respect to the frame field $\{X_{i}\}_{i=1}^{5}$,
all relations of \eqref{eqn:3.7} hold.

Then, taking in \eqref{eqn:2.19} that $(X,Y)=(X_1,X_4)$ and making use of \eqref{eqn:3.7}, we get
$$
0=(\lambda-\alpha)g(G(X_1,X_2),U).
$$
It follows that $g(G(X_1,X_2),U)=0$ and $G(X_1,X_2)=\pm\sqrt{(1-\theta^2)/3}\,\xi$.

In case $G(X_1,X_2)=-\sqrt{(1-\theta^2)/3}\,\xi$, with respect to the normal vector
$\tilde{\xi}=-\xi$, we have $G(X_1,X_2)=\sqrt{(1-\theta^2)/3}\,\tilde\xi$, and
the principal curvatures become $\tilde{\alpha}=-\alpha$, $\tilde{\lambda}=-\lambda$,
$\tilde{\beta}=-\beta$, and $X_1,X_5\in V_{\tilde{\alpha}}$, $X_2\in V_{\tilde{\beta}}$,
$X_3,X_4\in V_{\tilde{\lambda}}$. So it is sufficient to show that $G(X_1,X_2)=\sqrt{(1-\theta^2)/3}\,\xi$.

Taking in \eqref{eqn:2.19}, respectively, $(X,Y)=(X_1,X_2), (X_1,X_3), (X_2,X_4), (X_3,X_4)$,
and making use of \eqref{eqn:3.7}, we have
\begin{equation}\label{eqn:5.3}
-\tfrac{\theta}{6}=(\alpha-\beta)\sqrt{\tfrac{1-\theta^2}{3}}
+(\alpha^2-\alpha\beta)\theta,
\end{equation}
\begin{equation}\label{eqn:5.4}
\sqrt{3}\alpha+\tfrac{\sqrt{3}}{6(\alpha-\lambda)}=\tfrac{\theta}{\sqrt{1-\theta^2}},
\end{equation}
\begin{equation}\label{eqn:5.5}
-\tfrac{\sqrt{1-\theta^2}}{6}=-\tfrac{\sqrt{3}}{3}(2\alpha-\lambda-\beta)\theta
+(\alpha\lambda+\alpha\beta-2\lambda\beta)\sqrt{1-\theta^2},
\end{equation}
\begin{equation}\label{eqn:5.6}
-\sqrt{3}\lambda-\tfrac{\sqrt{3}}{12(\alpha-\lambda)}=\tfrac{\sqrt{1-\theta^2}}{\theta}.
\end{equation}

From these equations, we can derive a contradiction. Indeed, from \eqref{eqn:5.4} and
\eqref{eqn:5.6}, we have
\begin{equation}\label{eqn:5.7}
\sqrt{3}(\alpha-\lambda)+\tfrac{\sqrt{3}}{12(\alpha-\lambda)}=\tfrac{1}{\theta\sqrt{1-\theta^2}}.
\end{equation}
It follows that
$\alpha-\lambda=\tfrac{1\pm\sqrt{1-\theta^2+\theta^4}}{2\theta\sqrt{3(1-\theta^2)}}$. Then,
from  \eqref{eqn:5.4}, \eqref{eqn:5.6} and \eqref{eqn:5.3} we get
$$
\alpha=\tfrac{-1+\theta^2\pm\sqrt{1-\theta^2+\theta^4}}{\theta\sqrt{3(1-\theta^2)}},\ \
\lambda=\tfrac{-3+2\theta^2\pm\sqrt{1-\theta^2+\theta^4}}{2\theta\sqrt{3(1-\theta^2)}},\ \
\beta=\tfrac{\pm(2-\theta^2+\theta^4)-2(1-\theta^2)\sqrt{1-\theta^2+\theta^4}}
{2\sqrt{3}\theta\sqrt{(1-\theta^2)(1-\theta^2+\theta^4)}}.
$$

Now, substituting $\alpha,\lambda$ and $\beta$ into \eqref{eqn:5.5}, we get the
contradiction $\tfrac{\sqrt{1-\theta^2}}{3\sqrt{1-\theta^2+\theta^4}}=0$.

\vskip 1mm

{\bf (ii)-(b). $\theta=1$}.

In this subcase, both $(V_\alpha\cap\{U\}^\perp)\oplus V_\beta$ and $V_\lambda$ are $J$-invariant.
We take a local orthonormal frame field $\{X_{i}\}_{i=1}^{5}$ of $M$ such that
$$
AX_1=\alpha X_1,\ AX_2=\beta X_2,\ AX_3=\lambda X_3,\ AX_4=\lambda X_4,\ X_5=U,
$$
where $X_2=JX_1$ and $X_4=JX_3$. Then $G(X_1,X_3)\in {\rm Span}\{\xi,U\}$, and by \eqref{eqn:2.6},
we have $|G(X_1,X_3)|^2=\tfrac{1}{3}$. Taking in \eqref{eqn:2.19} $(X,Y)=(X_1,X_3)$ and $(X_1,X_4)$,
respectively, we easily get $(\alpha-\lambda)g(G(X_1,X_3),\xi)=(\alpha-\lambda)g(G(X_1,X_4),\xi)=0$.
This together with $G(X_1,X_4)=-JG(X_1,X_3)$ implies that $G(X_1,X_3)=0$, which is a contradiction.

\vskip 1mm

{\bf (ii)-(c). $\theta=0$}.

In this subcase, $J\{(V_\alpha\cap\{U\}^\perp)\oplus V_\beta\}=V_\lambda$. Then, we can take a
local orthonormal frame field $\{X_{i}\}_{i=1}^{5}$ of $M$ such that
$$
AX_1=\alpha X_1,\ AX_2=\lambda X_2,\ AX_3=\lambda X_3,\ AX_4=\beta X_4,\ X_5=U,
$$
where $X_2=JX_1$ and $X_4=JX_3$. Then $G(X_1,X_3)\in {\rm Span}\{\xi,U\}$ and $|G(X_1,X_3)|^2=\tfrac{1}{3}$.
Taking in \eqref{eqn:2.19} $(X,Y)=(X_1,X_3)$ and $(X_1,X_4)$, respectively, we get %
$$
(\alpha-\lambda)g(G(X_1,X_3),\xi)=(\alpha-\beta)g(G(X_1,X_4),\xi)=0.
$$
Then similar as the last subcase we get $G(X_1,X_3)=0$, which is a contradiction.

\vskip 1mm

{\bf (iii) $(\dim V_\alpha,\dim V_\lambda,\dim V_\beta)=(1,3,1)$ on $M$.}

In this case, we can take a local orthonormal frame field $\{X_{i}\}_{i=1}^{5}$ of $M$ such that
$$
AX_1=\beta X_1,\ AX_2=\lambda X_2,\ AX_3=\lambda X_3,\ AX_4=\lambda X_4,\ X_5=U,
$$
where $X_2=JX_1,X_4=JX_3$. Then $G(X_1,X_3)\in {\rm Span}\{\xi,U\}$ and $|G(X_1,X_3)|^2=\tfrac{1}{3}$.

Taking in \eqref{eqn:2.19} $(X,Y)=(X_1,X_2),(X_1,X_3)$ and $(X_1,X_4)$, respectively, we have
\begin{gather}
-\tfrac{1}{6}=\alpha\lambda+\alpha\beta-2\lambda\beta,\label{eqn:5.8}\\
(2\alpha-\lambda-\beta)g(G(X_1,X_3),\xi)=(2\alpha-\lambda-\beta)g(G(X_1,X_4),\xi)=0.\label{eqn:5.9}
\end{gather}
Then, by \eqref{eqn:5.9} and the fact $g(G(X_1,X_4),\xi)=g(-JG(X_1,X_3),\xi)=-g(G(X_1,X_3),U)$, we
get $2\alpha-\lambda-\beta=0$. This together with \eqref{eqn:5.8} gives the contradiction
$(\lambda-\beta)^2=-\tfrac{1}{3}$.

We have completed the proof of Lemma \ref{lemma:5.1}.
\end{proof}

Next, we shall determine the principal curvatures and show that they are constants.
Since we have the fact $\dim V_\alpha=1$ and $\dim V_\lambda=\dim V_\beta=2$,
without loss of generality, we shall assume that $\lambda>\beta$. Then, we can state our
result as follows:

\begin{lemma}\label{lemma:5.2}
All the three distinct principal curvatures $\alpha, \lambda$ and $\beta$ are constants.
More specifically, we have $\alpha=0,\ \lambda=\tfrac{\sqrt{1-\theta^2}+1}{2\sqrt{3}\theta}$
and $\beta=\tfrac{\sqrt{1-\theta^2}-1}{2\sqrt{3}\theta}$ for some $0<\theta\le1$.
\end{lemma}
\begin{proof}
It is easily seen that $|g(JX,Y)|$, for an orthonormal basis $\{X,Y\}$ of $V_\lambda$,
defines a well-defined function $\theta$ on $M$ satisfying $0\le\theta\le1$.
Since our concern is only local, in order to prove Lemma \ref{lemma:5.2},
by using the continuity of the principal curvature functions and $\theta$, we are sufficient to
consider the following three cases:

\vskip 1mm

{\bf (1). $0<\theta<1$ on $M$}.

In this case, we see that $JV_\lambda\neq V_\beta$ and $V_\lambda$ is not $J$-invariant. Then, we can
take a local orthonormal frame field $\{X_{i}\}_{i=1}^{5}$ of $M$ such that $\theta=g(JX_1,X_2)$ and
\begin{equation}\label{eqn:5.10}
AX_1=\lambda X_1,\ AX_2=\lambda X_2,\ AX_3=\beta X_3,\ AX_4=\beta X_4,\ AX_5=\alpha X_5,
\end{equation}
where $X_5=U,\ X_3=\tfrac{JX_1-\theta X_2}{\sqrt{1-\theta^2}},\ X_4=\tfrac{JX_2+\theta X_1}{\sqrt{1-\theta^2}}$.
Thus $G(X_1,X_2)\in {\rm Span}\{\xi,U\}$ and, by \eqref{eqn:2.6}, $|G(X_1,X_2)|^2=\tfrac13(1-\theta^2)$.
Moreover, it is easily seen that with respect to the frame field $\{X_{i}\}_{i=1}^{5}$, all relations of
\eqref{eqn:3.7} hold.

Taking, in \eqref{eqn:2.19}, $(X,Y)=(X_3,X_4)$ and $(X,Y)=(X_1,X_i)$ for $2\leq i\leq4$,
respectively, and making use of \eqref{eqn:3.7}, we have
\begin{equation}\label{eqn:5.11}
-\tfrac{\theta}{6}=2(\alpha-\lambda)g(G(X_1,X_2),\xi)+2\lambda(\alpha-\lambda)\theta,
\end{equation}
\begin{equation}\label{eqn:5.12}
-\tfrac16\sqrt{1-\theta^2}=-\tfrac{\theta(2\alpha-\lambda-\beta)}{\sqrt{1-\theta^2}}g(G(X_1,X_2),\xi)
+(\alpha\lambda+\alpha\beta-2\lambda\beta)\sqrt{1-\theta^2},
\end{equation}
\begin{equation}\label{eqn:5.13}
0=(2\alpha-\lambda-\beta)g(G(X_1,X_2),U),
\end{equation}
\begin{equation}\label{eqn:5.14}
\tfrac{\theta}{6}=-2(\alpha-\beta)g(G(X_1,X_2),\xi)+2\beta(\beta-\alpha)\theta.
\end{equation}

If $2\alpha-\lambda-\beta=0$, then together with \eqref{eqn:5.12} we derive a contradiction
$(\lambda-\beta)^2=-\tfrac{1}{3}$.

Hence $2\alpha-\lambda-\beta\neq0$. Then from \eqref{eqn:5.13} we get $g(G(X_1,X_2),U)=0$,
and therefore we obtain $G(X_1,X_2)=\pm\sqrt{(1-\theta^2)/3}\,\xi$. Without loss of generality,
we shall assume that $G(X_1,X_2)=-\sqrt{(1-\theta^2)/3}\,\xi$.

Actually, if it occurs $G(X_1,X_2)=\sqrt{(1-\theta^2)/3}\,\xi$, then $G(X_3,X_4)=-\sqrt{(1-\theta^2)/3}\,\xi$
and $g(JX_3,X_4)=-\theta<0$. Now, with respect to the normal vector field $\tilde{\xi}=-\xi$,
the principal curvatures become $\tilde{\alpha}=-\alpha$, $\tilde{\lambda}=-\beta$ and
$\tilde{\beta}=-\lambda$, $\tilde{\lambda}>\tilde{\beta}$.
Putting $\tilde{X_1}=X_3$, $\tilde{X_2}=-X_4$, $\tilde{X}_3=\tfrac{J\tilde{X}_1-\theta\tilde{X}_2}
{\sqrt{1-\theta^2}}$, $\tilde{X}_4=\tfrac{J\tilde{X}_2+\theta \tilde{X}_1}{\sqrt{1-\theta^2}}$ and
$\tilde{X}_5=U$, then, with respect to the orthonormal frame field $\{\tilde{X}_{i}\}_{i=1}^{5}$,
as assumed
we have $G(\tilde{X}_1,\tilde{X}_2)=-\sqrt{(1-\theta^2)/3}\,\tilde{\xi}$ and $g(J\tilde{X}_1,\tilde{X}_2)=\theta>0$.

Having the assumption $G(X_1,X_2)=-\sqrt{(1-\theta^2)/3}\,\xi$, the equations \eqref{eqn:5.11}, \eqref{eqn:5.12}
and \eqref{eqn:5.14} become
\begin{equation}\label{eqn:5.15}
\theta=4\sqrt{3}(\alpha-\lambda)\sqrt{1-\theta^2}+12\lambda(\lambda-\alpha)\theta,
\end{equation}
\begin{equation}\label{eqn:5.16}
-\sqrt{1-\theta^2}=2\sqrt{3}\theta(2\alpha-\lambda-\beta)
+6(\alpha\lambda+\alpha\beta-2\lambda\beta)\sqrt{1-\theta^2},
\end{equation}
\begin{equation}\label{eqn:5.17}
\theta=4\sqrt{3}(\alpha-\beta)\sqrt{1-\theta^2}+12\beta(\beta-\alpha)\theta.
\end{equation}
Then, solving $\lambda$ and $\beta$ from \eqref{eqn:5.15} and \eqref{eqn:5.17}, we obtain
$$
\lambda+\beta=\tfrac{3\alpha\theta+\sqrt{3(1-\theta^2)}}{3\theta},\ \ \lambda\beta=\tfrac{4\alpha\sqrt{3(1-\theta^2)}-\theta}{12\theta}.
$$
This combining with \eqref{eqn:5.16} gives $\alpha(\alpha\sqrt{1-\theta^2}+\tfrac{2\theta^2-1}{\sqrt{3}\theta})=0$.
Hence, $\alpha=0$ or $\alpha=\tfrac{1-2\theta^2}{\theta\sqrt{3-3\theta^2}}$.

In conclusion, we can solve the above equations to obtain two possibilities:

{\bf Case (1)-(i)}: $\alpha=0,\ \ \lambda=\tfrac{\sqrt{1-\theta^2}+1}{2\sqrt{3}\theta},\ \
\beta=\tfrac{\sqrt{1-\theta^2}-1}{2\sqrt{3}\theta}$;

{\bf Case (1)-(ii)}: $\alpha=\tfrac{1-2\theta^2}{\theta\sqrt{3(1-\theta^2)}},\ \
\lambda=\tfrac{2-3\theta^2+\theta}{2\theta\sqrt{3(1-\theta^2)}},\ \
\beta=\tfrac{2-3\theta^2-\theta}{2\theta\sqrt{3(1-\theta^2)}}$.

\vskip 1mm

Before dealing with these two subcases in more details, we need some preparations.

Put $\nabla_{X_i}X_j=\sum \Gamma_{ij}^{k}X_{k}$ with $\Gamma_{ij}^{k}=-\Gamma_{ik}^{j}$, $1\le i,j,k\le 5$.
First of all, we have
\begin{equation*}
G(X_1,\xi)=-\sum_{i=1}^5 \Gamma_{15}^{i}X_i+\lambda JX_1.
\end{equation*}
On the other hand, the facts $g(G(X_1,X_2),\xi)=-\sqrt{(1-\theta^2)/3}$ and $g(G(X_1,X_2),U)=0$
imply that $G(X_1,\xi)=\sqrt{\tfrac{1-\theta^2}{3}}X_2-\tfrac{\sqrt{3}}{3}\theta X_3$. Hence,
we obtain
\begin{equation}\label{eqn:5.18}
\Gamma_{15}^{1}=0,\ \ \Gamma_{15}^{2}=\lambda\theta-\sqrt{\tfrac{1-\theta^2}{3}},\ \ \Gamma_{15}^{3}=\lambda\sqrt{1-\theta^2}+\tfrac{\sqrt{3}}{3}\theta,\ \ \Gamma_{15}^{4}=0.
\end{equation}

Similarly, calculating $G(X_i,\xi)$ for $2\leq i\leq 4$, we can further obtain
\begin{equation}\label{eqn:5.19}
\left\{
\begin{aligned}
&\Gamma_{25}^{1}=-\lambda\theta+\sqrt{\tfrac{1-\theta^2}{3}},\ \ \Gamma_{25}^{2}=0,\ \ \Gamma_{25}^{3}=0,\ \ \Gamma_{25}^{4}=\lambda\sqrt{1-\theta^2}+\tfrac{\sqrt{3}}{3}\theta,\\
&\Gamma_{35}^{1}=-\beta\sqrt{1-\theta^2}-\tfrac{\sqrt{3}}{3}\theta,\ \ \Gamma_{35}^{2}=0,\ \ \Gamma_{35}^{3}=0,\ \ \Gamma_{35}^{4}=-\beta\theta+\sqrt{\tfrac{1-\theta^2}{3}},\\
&\Gamma_{45}^{1}=0,\ \ \Gamma_{45}^{2}=-\beta\sqrt{1-\theta^2}-\tfrac{\sqrt{3}}{3}\theta,\ \ \Gamma_{45}^{3}=\beta\theta-\sqrt{\tfrac{1-\theta^2}{3}},\ \ \Gamma_{45}^{4}=0.
\end{aligned}\right.
\end{equation}

Now, we calculate $g((\nabla_{X_i}A){X_j}-(\nabla_{X_j}A){X_i},X_k)$ for each $1\leq i,j,k\leq4$.

First, by using \eqref{eqn:2.18} we easily see that $g((\nabla_{X_i} A){X_j}-(\nabla_{X_j} A){X_i},X_k)=0$.

On the other hand, by using \eqref{eqn:5.10} we can calculate $0=g((\nabla_{X_i} A){X_j}-(\nabla_{X_j} A){X_i},X_k)$
to conclude that $X_1\lambda=X_2\lambda=X_3\beta=X_4\beta=0$ that is $X_i\theta=0$ for $1\le i\le4$, and
$\Gamma_{ij}^k=\Gamma_{ik}^j=0$ for $i\in\{1,2,3,4\},j\in\{1,2\}$ and $k\in\{3,4\}$.

Next, by definition, the above information of $\{\Gamma_{ij}^k\}$ and \eqref{eqn:3.7}, we can get
$$
0=g(G(X_1,X_2),X_3)=g((\tilde{\nabla}_{X_1} J){X_2},X_3)=\sqrt{1-\theta^2}\,(\Gamma_{14}^3-\Gamma_{12}^1).
$$
It follows that $\Gamma_{14}^3=\Gamma_{12}^1$. Similarly, by calculating $0=g(G(X_i,X_1),X_4)$ for $2\le i\le4$,
we further get $\Gamma_{23}^4=\Gamma_{21}^2$, $\Gamma_{33}^4=\Gamma_{31}^2$ and $\Gamma_{43}^4=\Gamma_{41}^2$.

Moreover, by using \eqref{eqn:3.7} we have $g(G(U,X_1),X_4)=-\tfrac{\sqrt{3}}{3}$, then direct calculation
of its left hand side gives
\begin{equation}\label{eqn:5.20}
(\Gamma_{53}^4-\Gamma_{51}^2)\sqrt{1-\theta^2}+(\Gamma_{52}^4+
\Gamma_{51}^3)\theta=-\tfrac{\sqrt{3}}{3}.
\end{equation}

Finally, from now on we assume that $PX_i=\sum_{j=1}^4p_{ij}X_j$ for $1\leq i\leq4$, where $p_{ij}=p_{ji}$ and, by the
definition of $X_3$ and $X_4$, we have the following relations:
\begin{equation}\label{eqn:5.21}
\left\{
\begin{aligned}
&p_{23}=p_{14}-\tfrac{(p_{11}+p_{22})\theta}{\sqrt{1-\theta^2}},\ \
p_{33}=\tfrac{\theta^2p_{22}-p_{11}+2\theta^2p_{11}}{1-\theta^2}-\tfrac{2\theta p_{14}}{\sqrt{1-\theta^2}},\\
&p_{34}=\tfrac{(p_{13}-p_{24})\theta}{\sqrt{1-\theta^2}}-p_{12},\ \
p_{44}=\tfrac{2\theta p_{14}}{\sqrt{1-\theta^2}}-\tfrac{p_{22}+\theta^2 p_{11}}{1-\theta^2}.
\end{aligned}
\right.
\end{equation}

\vskip 1mm

Now, we come to discuss {\bf Case (1)-(i)} and show that in this subcase $\theta$ is constant.

For that purpose, we apply for the Codazzi equation \eqref{eqn:2.18} with $(X,Y)=(U,X_i)$ for
$1\leq i\leq4$, and then checking the results we obtain the following equations:
\begin{gather}
3U\lambda-p_{11}b-a(p_{12}\theta+p_{13}\sqrt{1-\theta^2}\,)=0,\label{eqn:5.22}\\
ap_{11}\theta-p_{12}b-ap_{14}\sqrt{1-\theta^2}=0,\label{eqn:5.23}\\
2ap_{14}\theta-1-2p_{13}b+\tfrac{2\sqrt{3}}{\theta}\Gamma_{51}^3+2ap_{11}\sqrt{1-\theta^2}=0,\label{eqn:5.24}\\
\tfrac{\sqrt{3}}{\theta}\Gamma_{51}^4-p_{14}b-ap_{13}\theta+ap_{12}\sqrt{1-\theta^2}=0,\label{eqn:5.25}\\
3U\lambda-p_{22}b+ap_{12}\theta-ap_{24}\sqrt{1-\theta^2}=0,\label{eqn:5.26}
\end{gather}
\begin{equation}\label{eqn:5.27}
\begin{aligned}
\Gamma_{52}^3\sqrt{3(1-\theta^2)}&+b\theta\big[(p_{11}+p_{22})\theta-p_{14}\sqrt{1-\theta^2}\,\big]\\[-1mm]
&+a\theta(p_{12}-p_{12}\theta^2+p_{24}\theta\sqrt{1-\theta^2}\,)=0,
\end{aligned}
\end{equation}
\begin{equation}\label{eqn:5.28}
\begin{aligned}
2\sqrt{3}\Gamma_{52}^4(\theta^2-1)-\theta\big\{&\theta^2-1+2p_{24}b(\theta^2-1)\\[-1mm]
&+2a\big[p_{14}\theta(\theta^2-1)+\sqrt{1-\theta^2}(p_{22}+p_{11}\theta^2)\big]\big\}=0,
\end{aligned}
\end{equation}
\begin{equation}\label{eqn:5.29}
\begin{aligned}
2p_{14}b\theta(\theta^2-1)&+\sqrt{1-\theta^2}\big[(2p_{11}b+p_{22}b+3U\beta)\theta^2-p_{11}b-3U\beta\big]\\[-1mm]
&+a(\theta^2-1)\big[p_{13}-\theta(p_{24}\theta+p_{12}\sqrt{1-\theta^2}\,)\big]=0,
\end{aligned}
\end{equation}
\begin{equation}\label{eqn:5.30}
\begin{aligned}
b(\theta^2-1)&\big[(p_{24}-p_{13})\theta+p_{12}\sqrt{1-\theta^2}\,\big]\\[-1mm]
&+a\big[\theta\sqrt{1-\theta^2}(p_{22}+p_{11}\theta^2)+p_{14}(\theta^4-1)\big]=0,
\end{aligned}
\end{equation}
\begin{equation}\label{eqn:5.31}
\begin{aligned}
2p_{14}b\theta(\theta^2-1)&+\sqrt{1-\theta^2}\big[p_{22}b+3U\beta+(p_{11}b-3U\beta)\theta^2\big]\\[-1mm]
&-a(\theta^2-1)\big[p_{24}+\theta(p_{12}\sqrt{1-\theta^2}-p_{13}\theta\,)\big]=0.
\end{aligned}
\end{equation}

Calculating \eqref{eqn:5.22}\,-\,\eqref{eqn:5.26} and \eqref{eqn:5.29}+\eqref{eqn:5.31}, respectively,
we obtain
\begin{equation}\label{eqn:5.32}
0=(p_{22}-p_{11})b+a\big[-2p_{12}\theta+(p_{24}-p_{13})\sqrt{1-\theta^2}\,\big],
\end{equation}
\begin{equation}\label{eqn:5.33}
\begin{split}
0=&a(1-\theta^2)\big[(p_{24}-p_{13})(1+\theta^2)+2p_{12}\theta\sqrt{1-\theta^2}\,\big]\\[-1mm]
&+b\big\{4p_{14}\theta(\theta^2-1)+\sqrt{1-\theta^2}\big[p_{22}-p_{11}+(3p_{11}+p_{22})\theta^2\big]\big\}.
\end{split}
\end{equation}

Now, we claim that $a\not=0$ holds on $M$.

Indeed, if otherwise, we assume $a(z)=0$ for some $z\in M$. Then, carrying calculations below at $z$, we
have $b=\pm1$ and, by \eqref{eqn:5.32}, \eqref{eqn:5.33}, \eqref{eqn:5.23} and \eqref{eqn:5.30}, we have
\begin{equation}\label{eqn:5.34}
p_{22}-p_{11}=p_{12}=p_{24}-p_{13}=0,\ \ p_{14}=\tfrac{p_{11}\theta}{\sqrt{1-\theta^2}}.
\end{equation}
From \eqref{eqn:5.22} and \eqref{eqn:5.31}, we obtain $U\lambda=-U\beta=\tfrac13p_{11}b$ and thus
$U(\lambda+\beta)=0$. Then, as $\lambda+\beta=\tfrac{\sqrt{1-\theta^2}}{\sqrt{3}\theta}$ and $0<\theta<1$,
we get $U\theta=0$ and thus $U\lambda=U\beta=p_{11}=0$. From \eqref{eqn:5.34}, we have
$p_{11}=p_{12}=p_{22}=p_{14}=0$.

Finally, we apply for $0=g(G(PX_1,PX_2)+PG(X_1,X_2),U)$. By direct calculation of the right hand side,
making use of the fact $G(X_1,X_2)=-\sqrt{\tfrac{1-\theta^2}{3}}\,\xi$, \eqref{eqn:3.7} and \eqref{eqn:5.21},
we get the contradiction $\sqrt{1-\theta^2}b=0$, which verifies the claim.

\vskip 2mm

As $a\not=0$, from \eqref{eqn:5.23} we solve $p_{14}=\tfrac{ap_{11}\theta-p_{12}b}{a\sqrt{1-\theta^2}}$.
Then, from \eqref{eqn:5.32}, \eqref{eqn:5.33} and \eqref{eqn:5.30}, we obtain a
matrix equation $AB=0$, where
$$
A=(p_{22}-p_{11},\ p_{12}, \ p_{24}-p_{13}),
$$
$$
B=\left(
  \begin{array}{ccc}
    b & b(1+\theta^2) & -a\\
    -2a\theta & \tfrac{4b^2\theta+2a^2\theta(1-\theta^2)}{a} & -2b\theta \\[1mm]
    a\sqrt{1-\theta^2} & a\sqrt{1-\theta^2}(1+\theta^2)& b\sqrt{1-\theta^2}
  \end{array}
\right).
$$

The fact $\det B=\tfrac{4\theta\sqrt{1-\theta^2}}{a}\neq0$ implies that
$p_{22}-p_{11}=p_{12}=p_{24}-p_{13}=0$. By \eqref{eqn:5.22} and \eqref{eqn:5.31},
we have $U\lambda=-U\beta=\tfrac13(p_{11}b+ap_{13}\sqrt{1-\theta^2}\,)$.
The fact $0<\theta<1$ and $\lambda+\beta=\tfrac{\sqrt{1-\theta^2}}{\sqrt{3}\theta}$
then implies that $U\theta=0$.
This combining with $X_i\lambda=X_i\beta=0$ for $1\leq i\leq4$ shows that $\theta$ and
so that $\lambda$ and $\beta$ are constants on $M$.

Moreover, from \eqref{eqn:5.22} up to \eqref{eqn:5.31}, we can finally obtain:
\begin{equation}\label{eqn:5.35}
p_{13}=-\tfrac{p_{11}b}{a\sqrt{1-\theta^2}},\ \ p_{14}=\tfrac{p_{11}\theta}{\sqrt{1-\theta^2}},\ \ \Gamma_{51}^3=\Gamma_{52}^4=\tfrac{\theta(-2p_{11}+a\sqrt{1-\theta^2})}{2a\sqrt{3-3\theta^2}},\ \
\Gamma_{51}^4=\Gamma_{52}^3=0.
\end{equation}
Then, by $\sum_{i=1}^4 (p_{1i})^2=1$, we get $(p_{11})^2=a^2(1-\theta^2)$.

Now, calculating the curvature tensor, we obtain
\begin{equation}\label{eqn:5.36}
\begin{aligned}
g(R(X_1,X_3)X_3,X_1)&=\Gamma_{31}^5\Gamma_{53}^1-\Gamma_{13}^5\Gamma_{35}^1-\Gamma_{13}^5\Gamma_{53}^1
=\tfrac{4p_{11}(1+\theta^2)-a\sqrt{1-\theta^2}(5+3\theta^2)}{12a\sqrt{1-\theta^2}}.
\end{aligned}
\end{equation}

On the other hand, by Gauss equation \eqref{eqn:2.17} and the fact $a^2+b^2=1$, we have
\begin{equation}\label{eqn:5.37}
g(R(X_1,X_3)X_3,X_1)=\tfrac{a^2(10\theta^2-7-3\theta^4)-4(p_{11})^2(\theta^2-2)}{12a^2(\theta^2-1)}.
\end{equation}
Comparing these two calculations, we get
$$
(p_{11})^2(2-\theta^2)+3a^2(\theta^2-1)+ap_{11}\sqrt{1-\theta^2}(1+\theta^2)=0.
$$
Then, by using $(p_{11})^2=a^2(1-\theta^2)$, we finally get $p_{11}=a\sqrt{1-\theta^2}$.
It follows that, by \eqref{eqn:5.20}, \eqref{eqn:5.35} and the previous results about $p_{ij}$, we have
\begin{equation}\label{eqn:5.38}
\left\{
\begin{aligned}
&p_{11}=p_{22}=-p_{33}=-p_{44}=a\sqrt{1-\theta^2},\ \ p_{12}=p_{34}=0,\\[-1mm]
&p_{13}=p_{24}=-b,\ p_{14}=-p_{23}=a\theta,\\[-1mm]
&\Gamma_{51}^3=\Gamma_{52}^4=-\tfrac{\theta}{2\sqrt{3}},\ \ \Gamma_{53}^4=\Gamma_{51}^2-\sqrt{\tfrac{1-\theta^2}{3}}.
\end{aligned}\right.
\end{equation}

Later, in Lemma \ref{lemma:5.3}, we will show that {\bf Case (1)-(ii)} occurs only if
$\theta=\tfrac{\sqrt{2}}{2}$. But this implies that {\bf Case (1)-(ii)} is actually
a special situation of {\bf Case (1)-(i)}
with $\theta=\tfrac{\sqrt{2}}{2}$.

\vskip 1mm

{\bf (2). $\theta=1$ on $M$}.

In this case, it is easy to see that $M$ satisfies $A\phi=\phi A$. According to
Proposition 5.7 of \cite{H-Y-Z}, the principal curvatures of $M$ are $\alpha=0$,
$\lambda=\tfrac{\sqrt{3}}{6}$ and $\beta=-\tfrac{\sqrt{3}}{6}$. This exactly
shows that expressions of the principal
curvatures stated in {\bf Case (1)-(i)} are valid also for $\theta=1$.

\vskip 1mm

{\bf (3). $\theta=0$ on $M$}.

In this case, we choose a local orthonormal frame field $\{X_{i}\}_{i=1}^{5}$ of $M$ such that
$$
AX_1=\lambda X_1,\ AX_2=\beta X_2,\ AX_3=\lambda X_3,\ AX_4=\beta X_4,\ X_5=U,
$$
where $X_2=JX_1$ and $X_4=JX_3$. Then $G(X_1,X_3)\in {\rm Span}\{\xi,U\}$
and $|G(X_1,X_3)|^2=\tfrac{1}{3}$.
Now, taking in \eqref{eqn:2.19} $(X,Y)=(X_1,X_2)$, $(X_1,X_3)$ and $(X_1,X_4)$, respectively, we obtain
\begin{equation}\label{eqn:5.39}
\alpha\beta+\alpha\lambda-2\lambda\beta=-\tfrac{1}{6},
\end{equation}
\begin{equation}\label{eqn:5.40}
(\alpha-\lambda)g(G(X_1,X_3),\xi)=0,\ \
(2\alpha-\lambda-\beta)g(G(X_1,X_3),U)=0.
\end{equation}

From \eqref{eqn:5.40}, $\alpha\neq\lambda$ and $|G(X_1,X_3)|^2=\tfrac{1}{3}$, we get $2\alpha-\lambda-\beta=0$.
This combining with \eqref{eqn:5.39} gives the contradiction $(\lambda-\beta)^2=-\tfrac{1}{3}$.

We have completed the proof of Lemma \ref{lemma:5.2}.
\end{proof}

\begin{lemma}\label{lemma:5.3}
If {\bf Case (1)-(ii)} in the proof of Lemma \ref{lemma:5.2} does occur, then
$\theta=\tfrac{\sqrt{2}}2$.
\end{lemma}
\begin{proof}
First of all, according to Lemma \ref{lemma:2.2}, $\alpha$ is constant. Hence,
by the formulas for {\bf Case (1)-(ii)} of the proof of Lemma
\ref{lemma:5.2}, also $\theta, \lambda$ and $\beta$
are constants. Now, since the local orthonormal frame field $\{X_{i}\}_{i=1}^{5}$ of $M$
satisfy \eqref{eqn:5.10}, we apply for the Codazzi equation \eqref{eqn:2.18} with
$(X,Y)=(U,X_i)$ for $1\leq i\leq4$. Then, by checking the results, as in {\bf Case (1)-(i)}
we obtain the equations \eqref{eqn:5.22}, \eqref{eqn:5.23}, \eqref{eqn:5.26} and
\eqref{eqn:5.29}--\eqref{eqn:5.31} with $U\lambda=U\beta=0$. Moreover, we have the
following additional four equations:
\begin{gather}
\theta\big\{2\sqrt{3}\Gamma_{51}^3+\theta-2p_{13}b\sqrt{1-\theta^2}
+2a\big[p_{11}(1-\theta^2)+p_{14}\theta\sqrt{1-\theta^2}\,\big]\big\}-1=0,\label{eqn:5.41}\\
\sqrt{3}\Gamma_{51}^4-p_{14}b\sqrt{1-\theta^2}+
a\big[p_{12}(1-\theta^2)-p_{13}\theta\sqrt{1-\theta^2}\,\big]=0,\label{eqn:5.42}\\
\sqrt{3}\Gamma_{52}^3+(p_{11}+p_{22})b\theta-p_{14}b\sqrt{1-\theta^2}
+a\big[p_{12}(1-\theta^2)+p_{24}\theta\sqrt{1-\theta^2}\,\big]=0,\label{eqn:5.43}\\
\theta\big\{2\sqrt{3}\Gamma_{52}^4+\theta-2p_{24}b\sqrt{1-\theta^2}
+2a\big[p_{22}+\theta(p_{11}\theta-p_{14}\sqrt{1-\theta^2}\,)\big]\big\}-1=0.\label{eqn:5.44}
\end{gather}
It follows that \eqref{eqn:5.32} and \eqref{eqn:5.33} are still valid. Then, similar discussions
as in dealing with {\bf Case (1)-(i)}, we have
$$
a\neq0,\ (p_{11})^2=a^2(1-\theta^2),\ p_{22}=p_{11},\ p_{12}=0,\ p_{13}=p_{24}=-\tfrac{p_{11}b}{a\sqrt{1-\theta^2}},
\ p_{14}=\tfrac{p_{11}\theta}{\sqrt{1-\theta^2}}.
$$
Moreover, by using the equations \eqref{eqn:5.41}\,--\,\eqref{eqn:5.44}, we can get
$$
\Gamma_{51}^3=\Gamma_{52}^4=\tfrac{a-2p_{11}\theta-a\theta^2}{2\sqrt{3}a\theta},\ \ \Gamma_{51}^4=\Gamma_{52}^3=0.
$$

Now, calculating the curvature tensor, we obtain
$$
\begin{aligned}
&g(R(X_1,X_3)X_4,X_2)=\Gamma_{34}^5\Gamma_{15}^2-\Gamma_{13}^5\Gamma_{54}^2+\Gamma_{31}^5\Gamma_{54}^2
=\tfrac{a(6\theta^2-4-3\theta^4)-4p_{11}\theta(\theta^2-2)}{12a\theta^2},\\
&g(R(X_1,X_3)X_3,X_1)=\Gamma_{31}^5\Gamma_{53}^1-\Gamma_{13}^5\Gamma_{35}^1-\Gamma_{13}^5\Gamma_{53}^1
=\tfrac{a(11\theta^2-8-3\theta^4)-4p_{11}\theta(\theta^2-2)}{12a\theta^2}.
\end{aligned}
$$

On the other hand, by the Gauss equation \eqref{eqn:2.17} and the fact $a^2+b^2=1$, we have
$$
\begin{aligned}
&g(R(X_1,X_3)X_4,X_2)=\tfrac{4(p_{11})^2\theta^2+a^2(2-3\theta^2)(1-\theta^2)}{12a^2(1-\theta^2)},\\
&g(R(X_1,X_3)X_3,X_1)=\tfrac{4(p_{11})^2\theta^2(2-\theta^2)-a^2(1-\theta^2)^2(4+3\theta^2)}{12a^2\theta^2(\theta^2-1)}.
\end{aligned}
$$
Comparing these two calculations, respectively, we can obtain
\begin{equation}\label{eqn:5.45}
(p_{11})^2\theta^4-ap_{11}\theta(2-\theta^2)(1-\theta^2)+a^2(1-\theta^2)^2=0,
\end{equation}
\begin{equation}\label{eqn:5.46}
(p_{11})^2\theta^2(\theta^2-2)-ap_{11}\theta(2-\theta^2)(1-\theta^2)+3a^2(1-\theta^2)^2=0.
\end{equation}
Now the calculation \eqref{eqn:5.45}-\eqref{eqn:5.46} gives that
$$
(p_{11})^2\theta^2=a^2(1-\theta^2)^2,
$$
and, by using the fact $(p_{11})^2=a^2(1-\theta^2)$, we obtain $\theta=\tfrac{\sqrt{2}}{2}$.

This completes the proof of Lemma \ref{lemma:5.3}.
\end{proof}

Based on Lemma \ref{lemma:5.2}, we can prove the following result for Hopf hypersurfaces which
is an interesting counterpart of Proposition 5.8 in \cite{H-Y-Z}.
\begin{proposition}\label{prop:5.1}
Let $M$ be a Hopf hypersurface of the NK $\mathbf{S}^3\times\mathbf{S}^3$ with three distinct principal
curvatures and assume that the almost product structure $P$ of $M$ preserves the holomorphic distribution,
i.e., $P\{U\}^\perp=\{U\}^\perp$. Then either
$P\xi=\tfrac12\xi+\tfrac{\sqrt{3}}{2}J\xi$, or $P\xi=\tfrac12\xi-\tfrac{\sqrt{3}}{2}J\xi$, or
$P\xi=-\xi$.
\end{proposition}
\begin{proof}
We first assume that $0<\theta<1$. Let $\{X_i\}_{i=1}^5$ be as described by \eqref{eqn:5.10}.
Then, by using \eqref{eqn:3.7}, \eqref{eqn:5.38} and the fact $G(X_1,X_2)=-\sqrt{(1-\theta^2)/3}\,\xi$,
we can show that the equation $0=g(G(PX_1,PX_2)+PG(X_1,X_2),\xi)$ becomes equivalently
$$
(1-2a)(1+a)=0.
$$
This implies the assertion that we have three possibilities for $P\xi$, namely,

(1) $a=\tfrac{1}{2}$ and $b=-\tfrac{\sqrt{3}}{2}$, \ (2) $a=\tfrac{1}{2}$ and $b=\tfrac{\sqrt{3}}{2}$, \
(3) $a=-1$ and $b=0$.

\vskip 1mm

Next, if $\theta=1$, then as stated before the hypersurface satisfies $A\phi=\phi A$ and the assertion
follows from Proposition 5.8 of \cite{H-Y-Z}.
\end{proof}

\vskip 1mm

For the sake of later's purpose, we summarize the following conclusion that we have established.
\begin{lemma}\label{lemma:5.4}
For $0<\theta<1$ with $\alpha=0,\ \lambda=\tfrac{\sqrt{1-\theta^2}+1}{2\sqrt{3}\theta}$
and $\beta=\tfrac{\sqrt{1-\theta^2}-1}{2\sqrt{3}\theta}$, the vector $P\xi$ has three possibilities:
$\tfrac12\xi+\tfrac{\sqrt{3}}{2}J\xi,\ \tfrac12\xi-\tfrac{\sqrt{3}}{2}J\xi,\ -\xi$.
For each of these cases, we have a local orthonormal frame $\{X_i\}_{i=1}^5$, which is
described by \eqref{eqn:5.10}, such that $PX_i=\sum_{j=1}^4p_{ij}X_j$ for $1\leq i\leq4$, and $\{p_{ij}\}$
satisfy \eqref{eqn:5.38}. Moreover, with respect to $\{X_i\}_{i=1}^5$, the connection
coefficients $\{\Gamma_{ij}^k\}$ satisfy \eqref{eqn:5.18}, \eqref{eqn:5.19}, \eqref{eqn:5.38}, as well as
the following relations:
\begin{equation}\label{eqn:5.47}
\left\{
\begin{aligned}
&\Gamma_{ij}^k=0,\ {\rm if}\ i\in\{1,2,3,4\},\ j\in\{1,2\},\ k\in\{3,4\};\\
&\Gamma_{14}^3=\Gamma_{12}^1,\ \Gamma_{23}^4=\Gamma_{21}^2,\ \Gamma_{33}^4=\Gamma_{31}^2,\ \Gamma_{43}^4=\Gamma_{41}^2,\ \Gamma_{51}^4=\Gamma_{52}^3=0.
\end{aligned}\right.
\end{equation}
\end{lemma}

\vskip 2mm

\subsection{Proof of Theorem \ref{thm:1.2}}\label{sect:5.3}~

We get the proof of Theorem \ref{thm:1.2} as a direct consequence of three results
concerning the three possibilities for $P\xi$ described in Proposition \ref{prop:5.1}.
First of all, we prove the following result:

\begin{theorem}\label{thm:5.1}
Let $M$ be a Hopf hypersurface of the NK $\mathbf{S}^3\times\mathbf{S}^3$ which
possesses three distinct principal curvatures and satisfies $P\{U\}^\perp=\{U\}^\perp$
on $M$. If $P\xi=\tfrac{1}{2}\xi+\tfrac{\sqrt{3}}{2}J\xi$, then, up to isometries
of type $\mathcal{F}_{abc}$, $M$ is locally given by
the embedding $f_r$ $(0<r\leq1)$ in Theorem \ref{thm:1.2}.
\end{theorem}

\begin{proof}
We first assume that $0<\theta<1$ and let $\{X_i\}_{i=1}^5$ be as described by \eqref{eqn:5.10}.
Put
\begin{equation}\label{eqn:5.48}
\left\{
\begin{aligned}
&\bar{e}_1=\tfrac{\sqrt{2+\sqrt{1-\theta^2}}}{2}X_1-\tfrac{\sqrt{3}}{2\sqrt{2+\sqrt{1-\theta^2}}}X_3
+\tfrac{\theta}{2\sqrt{2+\sqrt{1-\theta^2}}}X_4,\ \bar{e}_5=X_5=U,\\[-1mm]
&\bar{e}_2=\tfrac{\sqrt{2+\sqrt{1-\theta^2}}}{2}X_2-\tfrac{\theta}{2\sqrt{2+\sqrt{1-\theta^2}}}X_3
-\tfrac{\sqrt{3}}{2\sqrt{2+\sqrt{1-\theta^2}}}X_4,\\
&\bar{e}_3=\tfrac{\theta}{\sqrt{2+2\sqrt{1-\theta^2}}}X_2+\tfrac{\sqrt{1+\sqrt{1-\theta^2}}}{\sqrt{2}}X_3, \ \
\bar{e}_4=\tfrac{\theta}{\sqrt{2+2\sqrt{1-\theta^2}}}X_1-\tfrac{\sqrt{1+\sqrt{1-\theta^2}}}{\sqrt{2}}X_4.
\end{aligned}
\right.
\end{equation}
Then $\{\bar{e}_i\}_{i=1}^5$ is a local (non-orthonormal) frame field of $M$. We
consider the following decomposition of the tangent bundle of $M$:
$TM={\rm Span}\{\bar{e}_1,\bar{e}_2\}\oplus{\rm Span}\{\bar{e}_3,\bar{e}_4,\bar{e}_5\}$.

Using Lemma \ref{lemma:5.4}, we have
$$
\nabla_{\bar{e}_i} {\bar{e}_j}\in {\rm
Span}\{\bar{e}_1,\bar{e}_2,\bar{e}_5\}\ {\rm for}\ i,j=1,2; \ \ \
\nabla_{\bar{e}_i} {\bar{e}_j}\in {\rm
Span}\{\bar{e}_3,\bar{e}_4,\bar{e}_5\}\ {\rm for}\ i,j=3,4,5.
$$
Moreover, by direct calculation, we can show that
$$
[\bar{e}_i,\bar{e}_j]\in {\rm Span}\{\bar{e}_1,\bar{e}_2\}\ {\rm for}\ i,j=1,2;\ \ \
[\bar{e}_i,\bar{e}_j]\in {\rm Span}\{\bar{e}_3,\bar{e}_4,\bar{e}_5\}\ {\rm for}\ i,j=3,4,5.
$$
It follows that both $\rm{Span}\{\bar{e}_1,\bar{e}_2\}$ and $\rm{Span}\{\bar{e}_3,\bar{e}_4,\bar{e}_5\}$ are integrable
distributions. Let $M_1$ and $M_2$ be the integral manifolds of $\rm{Span}\{\bar{e}_3,\bar{e}_4,\bar{e}_5\}$
and $\rm{Span}\{\bar{e}_1,\bar{e}_2\}$, respectively. Note also that now we have
$$
g(A\bar{e}_i,\bar{e}_j)=0\ {\rm for}\ i,j=3,4,5;\ \ g(A\bar{e}_i,\bar{e}_j)=\tfrac{\sqrt{3(1-\theta^2)}}{4\theta}\delta_{ij}\ {\rm for}\ i,j=1,2.
$$
So we have $\tilde{\nabla}_{\bar{e}_i}{\bar{e}_j}\in TM_1\ {\rm for}\ i,j=3,4,5$; and $\tilde{\nabla}_{\bar{e}_i}{\bar{e}_j}=\hat{\nabla}_{\bar{e}_i}{\bar{e}_j}+\hat{h}(\bar{e}_i,\bar{e}_j)\ {\rm for}\ i,j=1,2$,
where $\hat{\nabla}$ is the Levi-Civita connection of $M_2$, and $\hat{h}$ is the second fundamental
form of the submanifold $M_2\hookrightarrow\mathbf{S}^3\times\mathbf{S}^3$.
Moreover, by direct calculations we can show that $\hat{h}(\bar{e}_i,\bar{e}_j)=(\tfrac{\sqrt{1-\theta^2}}
{4\theta}U+\tfrac{\sqrt{3(1-\theta^2)}}{4\theta}\xi)\delta_{ij},i,j=1,2$. Hence $M_1$ is a
totally geodesic submanifold of $\mathbf{S}^3\times\mathbf{S}^3$, whereas $M_2$ is a totally
umbilical submanifold of $\mathbf{S}^3\times\mathbf{S}^3$.

Applying for \eqref{eqn:2.12}, we further see that $M_1$ and $M_2$ have constant sectional curvature
$\tfrac{3}{4}$ and $\tfrac{1+2\theta^2}{4\theta^2}$, respectively. Thus, $M_1$ (resp.
$M_2$) is locally isometric to $\mathbf{S}^3$ (resp. $\mathbf{S}^2$)
equipped with metric $\tfrac{4}{3}g_0$ (resp. $\tfrac{4\theta^2}{1+2\theta^2}g_0$), where
$g_0$ denotes the standard metric of constant sectional curvature $1$ on $\mathbf{S}^3$ (resp. $\mathbf{S}^2$).
In particular, $M$ is locally diffeomorphic to the product manifold $\mathbf{S}^3\times\mathbf{S}^2$.

By the identification of $M$ with an open subset of $\mathbf{S}^3\times\mathbf{S}^2$,
we can express the hypersurface
$M$ by an immersion $f=(p,q)$ with the parametrization $(x,y)$ of $\mathbf{S}^3\times\mathbf{S}^2$
such that
$$
\begin{aligned}
f:\mathbf{S}^3\times\mathbf{S}^2
\longrightarrow\mathbf{S}^3\times\mathbf{S}^3,\ \ \
(x,y)\mapsto (p(x,y),q(x,y)).
\end{aligned}
$$

From \eqref{eqn:2.10}, $P\xi=\tfrac{1}{2}\xi-\tfrac{\sqrt{3}}{2}U$, \eqref{eqn:3.7}, \eqref{eqn:5.38} and \eqref{eqn:5.48}, it can be verified that
$$
Q\bar{e}_1=\bar{e}_1,\ \ Q\bar{e}_2=\bar{e}_2,\ \
Q\bar{e}_3=-\bar{e}_3,\ \ Qe_4=-\bar{e}_4,\ \ QU=-U.
$$
Then, by the definition of $Q$, it follows that $dp,dq:T(\mathbf{S}^3\times\mathbf{S}^2)\rightarrow
T\mathbf{S}^3$ have the following properties:
\begin{equation}\label{eqn:5.49}
\left\{
\begin{aligned}
(dp(v),0)&=\tfrac{1}{2}(df(v)-Qdf(v))=df(v),\\
(0,dq(v))&=\tfrac{1}{2}(df(v)+Qdf(v))=0,
\end{aligned}
\right.\ \ \ \ \forall\, v\in T(\mathbf{S}^3\times\{pt\}).
\end{equation}
\vskip-3mm
\begin{equation}\label{eqn:5.50}
\left\{
\begin{aligned}
(dp(w),0)&=\tfrac{1}{2}(df(w)-Qdf(w))=0,\\
(0,dq(w))&=\tfrac{1}{2}(df(w)+Qdf(w))=df(w),
\end{aligned}
\right.\ \forall\, w\in T(\{pt\}\times\mathbf{S}^2).
\end{equation}

The first equation of \eqref{eqn:5.50} shows that $p$ depends only on the first entry $x$,
and hence it can be regarded as a mapping from $\mathbf{S}^3$ to $\mathbf{S}^3$.
From \eqref{eqn:5.49} we see that $p:\mathbf{S}^3\to\mathbf{S}^3$ is a local diffeomorphism.
Noting that the pull-back metric $f^*g$ restricted on $\mathbf{S}^3\times \{pt\}$ is exactly
$\tfrac{4}{3}g_0$, $p$ is actually an isometry. By a
re-parametrization of the preimage $\mathbf{S}^3$, we can assume
that $p(x)=x$.

Similarly, from the second equation in \eqref{eqn:5.49} we derive that $q$ depends only on
the second entry $y$, thus $q$ is actually a mapping from $\mathbf{S}^2$ to $\mathbf{S}^3$.
As the second equation in \eqref{eqn:5.50} shows that $dq$ is of rank $2$, then $q(\mathbf{S}^2)$
is a $2$-dimensional submanifold in $\mathbf{S}^3$. Noting that the pull-back metric $f^*g$ restricted
on $\{pt\}\times\mathbf{S}^2$ is $\tfrac{4\theta^2}{1+2\theta^2}g_0$. It follows that
$\mathbf{S}^2$ is totally umbilical immersed in $\mathbf{S}^3$ and, up to an isometry of
$\mathbf{S}^3$, we can assume that $q(y)=\sqrt{1-r^2}+ry$, where
$r=\tfrac{\sqrt{3}\theta}{\sqrt{1+2\theta^2}}$ and $y\in\mathbf{S}^3\cap\mathrm{Im}\,\mathbb{H}$.

Hence, up to isometries of type $\mathcal{F}_{abc}$, $M$ is locally the image of the embedding $f_r$,
corresponding to $0<r<1$, as described in Theorem \ref{thm:1.2}.

\vskip 1mm

Next, we consider the case $\theta=1$. As we mentioned earlier, in this case $M$ satisfies
$A\phi=\phi A$. Then, according to Theorem 5.9 of \cite{H-Y-Z}, $M$ is locally given by the
embedding $f_1$ as described in Theorem \ref{thm:1.2}.

This completes the proof of Theorem \ref{thm:5.1}.
\end{proof}

\begin{theorem}\label{thm:5.2}
Let $M$ be a Hopf hypersurface of the NK $\mathbf{S}^3\times\mathbf{S}^3$ which
possesses three distinct principal curvatures and satisfies $P\{U\}^\perp=\{U\}^\perp$
on $M$. If $P\xi=\tfrac{1}{2}\xi-\tfrac{\sqrt{3}}{2}J\xi$, then, up to isometries
of type $\mathcal{F}_{abc}$, $M$ is locally given by
the embedding $f_r'$ $(0<r\leq1)$ in Theorem \ref{thm:1.2}.
\end{theorem}
\begin{proof}
Given $M$, by using the isometry $\mathcal{F}_1$, we obviously get another Hopf hypersurface
$\mathcal{F}_1(M)$ of the NK $\mathbf{S}^3 \times \mathbf{S}^3$ which also possesses three
distinct principal curvatures.
From Theorem 5.1 of \cite{M-V}, we know that
the differential of the isometry $\mathcal{F}_1$ anticommutes with the almost complex
structure $J$, and commutes with the almost product structure $P$, that is,
$$
d\mathcal{F}_1\circ J=-J\circ d\mathcal{F}_1,\ \ \ \ d\mathcal{F}_1\circ P=
P\circ d\mathcal{F}_1.
$$

Noticing that $\xi':=d\mathcal{F}_1(\xi)$ and $U':=-J\xi'=-d\mathcal{F}_1(U)$ are the
unit normal vector field and the structure vector field of $\mathcal{F}_1(M)$.
By using $P\xi=\tfrac{1}{2}\xi-\tfrac{\sqrt{3}}{2}J\xi$, we have
\begin{equation*}
\begin{aligned}
P\xi'&=Pd\mathcal{F}_1(\xi)=d\mathcal{F}_1P(\xi)=d\mathcal{F}_1(\tfrac{1}{2}\xi-\tfrac{\sqrt{3}}{2}J\xi)\\
&=\tfrac{1}{2}d\mathcal{F}_1(\xi)+\tfrac{\sqrt{3}}{2}Jd\mathcal{F}_1(\xi)=\tfrac{1}{2}\xi'+\tfrac{\sqrt{3}}{2}J\xi'.
\end{aligned}
\end{equation*}
It follows that $P\{U'\}^\perp=\{U'\}^\perp$ holds on $\mathcal{F}_1(M)$.

Noticing that, for any unitary quaternions $a,b,c$, the isometries $\mathcal{F}_{abc}$
and $\mathcal{F}_1$ satisfy $(\mathcal{F}_1)^2={\rm id}$ and
$\mathcal{F}_{abc}\circ\mathcal{F}_1=\mathcal{F}_1\circ\mathcal{F}_{bac}$.
Then, applying for Theorem \ref{thm:5.1} to
the hypersurface $\mathcal{F}_1(M)$,
we immediately conclude the proof of Theorem \ref{thm:5.2}.
\end{proof}

\begin{theorem}\label{thm:5.3}
Let $M$ be a Hopf hypersurface of the NK $\mathbf{S}^3\times\mathbf{S}^3$ which
possesses three distinct principal curvatures and satisfies $P\{U\}^\perp=\{U\}^\perp$
on $M$. If $P\xi=-\xi$, then, up to isometries
of type $\mathcal{F}_{abc}$, $M$ is locally given by
the embedding $f_r''$ $(0<r\leq1)$ in Theorem \ref{thm:1.2}.
\end{theorem}

\begin{proof}
Given $M$, by using the isometry $\mathcal{F}_2$, we get another
Hopf hypersurface $\mathcal{F}_2(M)$ of the NK $\mathbf{S}^3\times\mathbf{S}^3$
which also possesses three distinct principal curvatures.
From Theorem 5.2 of \cite{M-V}, the differential of the isometry
$\mathcal{F}_2$ satisfies the following relationship with $J$ and $P$:
$$
d\mathcal{F}_2\circ J=-J\circ d\mathcal{F}_2,\ \ \ \
d\mathcal{F}_2\circ P=(-\tfrac{1}{2}P+\tfrac{\sqrt{3}}{2}JP)\circ d\mathcal{F}_2.
$$

Noticing that $\xi'':=d\mathcal{F}_2(\xi)$ and $U'':=-J\xi''=-d\mathcal{F}_2(U)$ are
the unit normal vector field and the structure vector field of $\mathcal{F}_2(M)$.
By using $P\xi=-\xi$, we have
\begin{equation*}
\begin{aligned}
P\xi''&=Pd\mathcal{F}_2(\xi)=-2d\mathcal{F}_2P(\xi)+\sqrt{3}JPd\mathcal{F}_2(\xi)\\
&=2d\mathcal{F}_2(\xi)+\sqrt{3}JP\xi''=2\xi''+\sqrt{3}JP\xi''.
\end{aligned}
\end{equation*}
It follows that
$P\xi''=\tfrac{1}{2}(\xi''-\sqrt{3}PJP\xi'')=\tfrac{1}{2}\xi''+\tfrac{\sqrt{3}}{2}J\xi''$,
and $P\{U''\}^\perp=\{U''\}^\perp$ holds on $\mathcal{F}_2(M)$.

Noticing also that, for any unitary quaternions $a,b,c$, the isometries $\mathcal{F}_{abc}$
and $\mathcal{F}_2$ satisfy $(\mathcal{F}_2)^2={\rm id}$ and
$\mathcal{F}_{abc} \circ \mathcal{F}_2=\mathcal{F}_2 \circ \mathcal{F}_{cba}$.
Then, applying for Theorem \ref{thm:5.1} to the hypersurface $\mathcal{F}_2(M)$,
we immediately conclude the proof of Theorem \ref{thm:5.3}.
\end{proof}

Finally, combining Proposition \ref{prop:5.1} and Theorems \ref{thm:5.1}--\ref{thm:5.3},
we have completed the proof of Theorem \ref{thm:1.2}.\qed

\vskip 2mm

\noindent{\bf Acknowledgements}. The authors are greatly indebted to
the referee for his/her carefully reading the first submitted version 
of this paper and giving elaborate comments and valuable suggestions
on revision so that the presentation can be greatly improved.

%%%%%%%%%%%%%%%%%%%%%%%%%%%%%%%%%%%%%%%%%%%%%%%%%%%%%%%%%%%%%%%%%%%%

\end{document}